\documentclass{amsart}

\usepackage{amsmath}
\usepackage{amsthm}
\usepackage{amssymb}
\usepackage{upgreek}
\usepackage{IEEEtrantools}

\usepackage{tikz}

\newcommand{\bb}[1]{\mathbb{#1}}
\newcommand{\cc}[1]{\mathcal{#1}}

\newcommand{\sminus}{\smallsetminus}
\newcommand{\Mod}{\!\!\pmod}

\DeclareMathOperator{\Ap}{Ap}

\newcommand{\C}{{\mathbb C}}

\newcommand{\Z}{{\mathbb Z}}
\newcommand{\N}{{\mathbb N}}
\newcommand{\R}{{\mathbb R}}

\usepackage{enumerate}
\let\oldenumerate\enumerate
\renewcommand{\enumerate}
	{
	\oldenumerate
	\setlength{\itemsep}{20pt}
	\setlength{\parskip}{2pt}
	\setlength{\parsep}{2pt}
	\setlength{\parindent}{1cm}
	}

\makeatletter
\let\@@pmod\pmod
\DeclareRobustCommand{\pmod}{\@ifstar\@pmods\@@pmod}
\def\@pmods#1{\mkern4mu({\operator@font mod}\mkern 6mu#1)}
\makeatother

\theoremstyle{plain}\newtheorem{proposition}{Proposition}[section]
	\newtheorem{theorem}[proposition]{Theorem}
	
	\newtheorem{cor}[proposition]{Corollary}
	\newtheorem{lem}[proposition]{Lemma}
	\newtheorem{conjecture}[proposition]{Conjecture}
	\newtheorem{problem}[proposition]{Problem}

\theoremstyle{remark}\newtheorem{example}[proposition]{Example}
	\newtheorem*{remark}{Remark}

\begin{document}

\title[Numerical Sets and Core Partitions]{Numerical Sets, Core Partitions, and Integer Points in Polytopes}

\author[H. Constantin]{Hannah Constantin}
\address{Department of Mathematics\\
University of Toronto\\
Toronto, ON M5S 2E4, Canada}
\email{hconstan@math.utoronto.ca}

\author[B. Houston-Edwards]{Ben Houston-Edwards}
\address{Department of Mathematics\\
Yale University\\
New Haven, CT 06511, USA}
\email{benjamin.houston-edwards@yale.edu}

\author[N. Kaplan]{Nathan Kaplan}
\address{Department of Mathematics\\
University of California\\
Irvine, CA 92697, USA}
\email{nckaplan@math.uci.edu}

\date{\today}

\begin{abstract}
We study a correspondence between numerical sets and integer partitions that leads to a bijection between simultaneous core partitions and the integer points of a certain polytope.  We use this correspondence to prove combinatorial results about core partitions.  For small values of $a$, we give formulas for the number of $(a,b)$-core partitions corresponding to numerical semigroups.  We also study the number of partitions with a given hook set.
\end{abstract}

\maketitle

%	We establish a one-to-one correspondence between numerical sets and integer partitions.
%	This leads to a bijection between $a$-core partitions and the points of $\bb{N}^{a-1}$, which in turn leads to a bijection between simultaneous $(a,b)$-core partitions and lattice points of an integer polytope in this space.
%	We use these correspondences to prove a few combinatorial results about $a$-core and $(a,b)$-core partitions.
%	For small values of $a$, we count the number of $(a,b)$-core partitions corresponding to numerical semigroups.
%	Using these techniques, we also study the number of partitions with a given hook set.
%\end{abstract}

\section{Introduction}

A large number of recent papers have studied statistical questions about sizes of simultaneous core partitions \cite{Aggarwal, Amdeberhan, Armstrong, Chen, Zeilberger, Johnson, Stanley, Thiel, Wang, Xiong1, Xiong2, Yang}.  One of the larger successes in this area is Johnson's proof of Armstrong's conjecture, which we state as Theorem \ref{ArmstrongsConj} below \cite{Johnson}.  Broadly, these problems address questions of the following type: Given a finite set of partitions, for example, the set of simultaneous $(a,b)$-core partitions, what can we say about the statistical properties of their sizes?  We use a correspondence between numerical sets and partitions to study these types of questions for partitions coming from families of numerical semigroups and for partitions with a fixed hook set.  

We first briefly introduce some notation necessary to explain our main results. A \emph{partition} $\lambda$ of $n$ is a sequence of positive integers $\lambda_1 \ge \lambda_2 \ge \cdots \ge  \lambda_k \ge 1$ whose sum is $n$.  We refer to the $\lambda_i$ as the \emph{parts} of the partition $\lambda$.  We represent a partition by its \emph{Young diagram}, a series of left aligned rows of boxes in which there $\lambda_i$ boxes in row $i$.  For any box of the Young diagram, its \emph{hook length} is the number of boxes directly to the right of it, plus the number of boxes directly below it, plus one for the box itself. 
%Figure~\ref{hookLengthPic} shows the hook lengths of the partition $\lambda = (4,2,2)$.
We denote by $H(\lambda)$ and $\cc H(\lambda)$ the \emph{hook set} and \emph{hook multiset} of $\lambda$---the set and multiset of hook lengths, respectively.

	\begin{figure}[h] \label{hookLengthPic}
	\caption{Young diagram and hook lengths of the partition $(4,2,2)$. This partition is both a $4$-core and a $7$-core.}
	\centering
			\begin{tikzpicture}[scale = 0.6]
			% \lambda = (4, 2, 2)		
			
				% 1 - hook lines
				\begin{scope}
					\foreach \x in {0,1,2,3}
						{ \draw (\x,2) rectangle ++(1,1);}
					\foreach \x in {0,1}
						{ \draw (\x,1) rectangle ++(1,1);}
					\foreach \x in {0,1}
						{ \draw (\x,0) rectangle ++(1,1);}
					
					\begin{scope}[blue]
						\node (a) at (0.5,1.5) {};
						\filldraw [radius=0.3] (a) circle;
						\filldraw [rounded corners =0.5mm] (a) -- ++(0,-0.1) rectangle ++(1.2,0.2);
						\filldraw [rounded corners = .5mm] (a) -- ++(-0.1,0) rectangle ++(0.2,-1.2);
					\end{scope}
				\end{scope}

				% 2 - one labelled
				\begin{scope}[xshift = 6cm]
					\foreach \x in {0,1,2,3}
						{ \draw (\x,2) rectangle ++(1,1);}
					\foreach \x in {0,1}
						{ \draw (\x,1) rectangle ++(1,1);}
					\foreach \x in {0,1}
						{ \draw (\x,0) rectangle ++(1,1);}
					
					\begin{scope}
						\node at (0.5,1.5) {3};
					\end{scope}
				\end{scope}

				% 3 - all labelled
				\begin{scope}[xshift = 12cm]
					\foreach \x in {0,1,2,3}
						{ \draw (\x,2) rectangle ++(1,1);}
					\foreach \x in {0,1}
						{ \draw (\x,1) rectangle ++(1,1);}
					\foreach \x in {0,1}
						{ \draw (\x,0) rectangle ++(1,1);}
					
					\begin{scope}
						\node at (0.5,2.5) {6};
						\node at (1.5,2.5) {5};
						\node at (2.5,2.5) {2};
						\node at (3.5,2.5) {1};
						
						\node at (0.5,1.5) {3};
						\node at (1.5,1.5) {2};
						
						\node at (0.5,0.5) {2};
						\node at (1.5,0.5) {1};
					\end{scope}
				\end{scope}

				% arrows
				\begin{scope}[ultra thick, dashed,  red, ->]
					\draw (3.5,1) -- (5.5,1);
					\draw (9.5,1) -- (11.5,1);
				\end{scope}
			\end{tikzpicture}
	\end{figure}
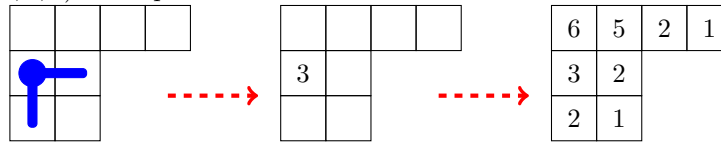

Hook lengths play an important role in the representation theory of the symmetric group.  For example, the Frame-Robinson-Thrall hook-length formula \cite{Frame}, expresses the dimension of the irreducible representation $\pi_\lambda$ of $S_n$ corresponding to a partition $\lambda$ of $n$:
	$$\dim \pi_\lambda = \frac{n!}{\prod_{h \in \cc{H}(\lambda)} h}.$$
A partition $\lambda$ of $n$ with no hook lengths divisible by $a$ is called an \emph{$a$-core partition} or more simply, an \emph{$a$-core}.  When $a$ is prime the corresponding irreducible representations have maximal $a$-adic valuation and play a role in the modular representation theory of $S_n$ \cite{Granville}.

There has been an explosion of recent papers studying enumerative questions about special classes of $a$-core partitions.  The set of $a$-cores is clearly infinite but the number of partitions that are both $a$-cores and $b$-cores, \emph{simultaneous $(a,b)$-cores}, is finite.  Similarly, an $(a_1,a_2,\ldots,a_k)$-core partition is an $a_i$-core for all $i \in [1,k]$.
There is a nice formula due to Anderson for the number of simultaneous $(a,b)$-core partitions that comes from establishing a bijection with a certain set of Dyck paths.

\begin{theorem}[Theorem 1 in \cite{Anderson}] \label{AndersonsTheorem}
	For coprime $a$ and $b$, the number of simultaneous $(a,b)$-core partitions is $\frac{1}{a + b} \binom{a+b}{a}$.
\end{theorem}

%More recently hook lengths and especially core partitions have been studied for their combinatorial properties.
%An \emph{$a$-core partition} is one in which no hook length is divisible by $a$.
%A \emph{simultaneous $(a,b)$-core partition} (or simply an \emph{$(a,b)$-core}) is a partition which is both an $a$-core and a $b$-core.
%Similarly an $(a_1,a_2,\ldots,a_n)$-core partition is an $a_i$-core for all $1\leq i \leq n$.
%For example, the partition in Figure~\ref{hookLengthPic} is a $(4,7)$-core partition.

%It is not hard to show there are an infinite number of $(a,b)$-core partitions if $\gcd(a,b) > 1$.
%By contrast, Anderson showed that for coprime $a$ and $b$, the set of $(a,b)$-cores is in bijection with Dyck paths from $(0,a)$ to $(b,0)$, which are counted by the Catalan numbers:

%\begin{theorem}[Anderson \cite{Anderson}] \label{AndersonsTheorem}
%	For coprime $a$ and $b$, the number of simultaneous $(a,b)$-core partitions is $\frac{1}{a + b} \binom{a+b}{a}$.
%\end{theorem}

It is natural to ask about the sizes of the partitions making up this finite set.  A formula for the size of the largest simultaneous $(a,b)$-core partition was first given by Olsson and Stanton.
\begin{theorem}[Theorem 4.1 in \cite{Olsson}]\label{knownLargestABcore}
For relatively prime positive integers $a$ and $b$, the largest $(a,b)$-core has size $(a^2-1)(b^2-1)/24$.  Moreover, there is a unique $(a,b)$-core of this size.
%For coprime $a$ and $b$, the unique largest $(a,b)$-core has size $(a^2-1)(b^2-1)/24$.
\end{theorem}
\noindent More recently, different proofs have been given by Tripathi \cite{Tripathi} and Johnson \cite{Johnson}.

%Because of its nice combinatorial proof and simple formula, this result has sparked much of the recent interest in the area.
%For example, the largest simultaneous $(a,b)$-core partition is known:

%\begin{theorem}[Olsson and Stanton \cite{Olsson}, Tripathi \cite{Tripathi}]
%	\label{knownLargestABcore}
%	For coprime $a$ and $b$, the unique largest $(a,b)$-core has size $(a^2-1)(b^2-1)/24$.
%\end{theorem}

In 2011 Armstrong gave a conjecture for the average size of an $(a,b)$-core partition that has a simple relation to the maximum size \cite{Armstrong}.  This conjecture was proven in the special case where $b = a+1$ by Stanley and Zanello \cite{Stanley} and more generally when $b \equiv 1 \pmod{a}$ by Aggarwal \cite{Aggarwal}.  The \emph{conjugate} of a partition $\lambda = (\lambda_1,\ldots, \lambda_k)$ is the partition $\widetilde\lambda = (\lambda_1',\ldots, \lambda_\ell')$ where $\lambda_j'$ is the number of parts of $\lambda \ge j$.  This is the partition we get by exchanging the rows and columns of the Young diagram of $\lambda$.  A partition is \emph{self-conjugate} if it is equal to its conjugate.  Armstrong's conjecture was proven for self-conjugate $(a,b)$-cores by Chen, Huang, and Wang \cite{Chen}.  After these partial results, the full theorem was proven by Johnson \cite{Johnson}.  Another proof was recently given by Wang \cite{Wang}.

%Armstrong conjectured the average size of an $(a,b)$-core partition, and after various partial results the theorem was finally proven recently by Johnson:
\begin{theorem}[Theorem 1.7 in \cite{Johnson}] \label{ArmstrongsConj}
For relatively prime positive integers $a$ and $b$, the average size of an $(a,b)$-core partition is $(a + b + 1)(a - 1)(b - 1)/24$.
\end{theorem}

Johnson's work \cite{Johnson} is of special interest to us because he proves Theorem \ref{ArmstrongsConj} by studying a bijection of $a$-core partitions with the lattice
	$$A_{a-1} =  \left\{(x_1,\ldots,x_a) \in \bb{Z}^a : \sum_{i=1}^a x_i = 0 \right\} $$
under which the simultaneous $(a,b)$-cores correspond to the integer points of a rational simplex. This bijection is given by the `signed abacus construction'.  Under this bijection, the size of a partition is given by a certain quadratic function \cite{Johnson}.  Johnson's works also gives the ability to compute higher moments of the distribution of the sizes of simultaneous $(a,b)$-cores, a problem also addressed in \cite{Zeilberger}.  Thiel and Williams also consider these higher moments and extend this approach to affine Weyl groups \cite{Thiel}.

Our approach is similar to Johnson's in that we study a correspondence between simultaneous $(a,b)$-core partitions and the integer points of a rational polytope, however we do not use the abacus construction.  Instead we study a bijection $\varphi$ between partitions and \emph{numerical sets}, subsets of $\N = \{0,1,2,\ldots\}$ that contain $0$ and have finite complement.  A numerical set that is closed under addition is called a \emph{numerical semigroup}. The bijection is given by considering the \emph{profile} of a partition $\lambda$, the sequence of southmost and eastmost edges of its Young diagram.  These steps are labeled by elements of $\N$ starting with the lower left corner of the Young diagram and moving to the upper right where the vertical steps exactly correspond to the elements of the complement of the associated numerical set.  This bijection is explained in detail in \cite{Keith} and is related to the Dyck path construction in \cite{Bras}.  The number of parts in the partition is equal to the size of the complement of the numerical set, and the hook set can be easily calculated from the numerical set.  Moreover, Keith and Nath use this bijection and basic facts about numerical sets to show the following.
\begin{theorem}[Theorem 1 in \cite{Keith}]\label{KeithNath}
Let $a_1, \ldots, a_k$ be distinct positive integers.  The number of simultaneous $(a_1,\ldots, a_k)$-cores is finite if and only if $\gcd\{a_1,\ldots, a_k\} = 1$.
\end{theorem}
Another proof of this result is given in \cite{Xiong1}.  In Section \ref{abCorePolytope} we show something stronger, that when $\gcd\{a_1,\ldots, a_k\} = 1$ the bijection $\varphi$ takes the set of simultaneous $(a_1,\ldots, a_k)$-cores to the lattice points of a rational polytope whose defining half-spaces we explicitly describe. In general it is still an open problem to give formulas for the number of such points in terms of $a_1,\ldots, a_k$.  The particular cases of $(s,s+1,s+2)$-cores and $(s,s+1,\ldots, s+k)$-cores have been addressed in \cite{Yang} and \cite{Amdeberhan}, respectively.

%In this paper we use a similar approach and study a different bijection of partitions with \emph{numerical sets}---subsets of $\bb{N}$ which contain 0 and have finite complement.
%The bijection is constructed by considering the rim path of the Young diagram of a partition.
%Horizontal steps correspond to elements of the associated numerical set, and vertical steps do not.
%We briefly introduce the bijection between lattice points of $\N^{a-1}$ and $a$-core partitions.

We use results of Marzuola and Miller about the \emph{atom monoid} associated to a numerical set \cite{Marzuola} along with the Kunz coordinate vector of a numerical semigroup, described by Blanco and Puerto in \cite{Blanco}, to give further bijections involving $a$-cores.  The atom monoid of a numerical set $T$ is defined by
	$$A(T) = \{n \in \bb{N} : n + T \subseteq T\}.$$
Note that $A(T) \subseteq T$ since $0 \in T$, and $A(T) = T$ if and only if $T$ is a numerical semigroup. The atom monoid is always closed under addition, so in some sense $A(T)$ is the underlying numerical semigroup of $T$. For a numerical semigroup $S$ containing $a$, the associated \emph{Ap\'ery tuple} is $\Ap(S) = (x_1,\ldots,x_{a-1}) \in \N^{a-1}$, where $ax_i + i$ is the smallest element of the numerical semigroup congruent to $i$ mod $a$. The Ap\'ery set of $S$ is $\{0,ax_1 + 1,\ldots, a x_{a-1} + a-1\}$. The definition of $\Ap(S)$ depends on $a$, but the specific value of $a$ we choose will always be clear from context.  We can directly extend this definition to numerical sets $T$ with $a \in A(T)$.  We summarize our bijections as a proposition that we prove in Section \ref{Correspondence}.
\begin{proposition}\label{bijections}
The bijection $\varphi$ described above gives a bijection between the set of $a$-core partitions and the set of numerical sets $T$ with $a \in A(T)$.  The map taking a numerical set $T$ with $a\in A(T)$ to its Ap\'ery tuple gives a bijection between these numerical sets and $\N^{a-1}$.
\end{proposition}

%We use the map $\varphi$ of the previous paragraph to give a bijection between $a$-core partitions and numerical sets $S$ whose \emph{atom monoid}, 
%Using atom monoids and Ap\'ery tuples (also known as Kunz coordinate vectors), we can extend this correspondence to find a bijection of $a$-core partitions with $\bb{N}^{a-1}$.
%In \cite{Marzuola}, Marzuola and Miller introduce the \emph{atom monoid} of a numerical set as its ``largest'' underlying numerical semigroup (a numerical set closed under addition).
%For any $a$ in a numerical semigroup, the associated \emph{Ap\'ery tuple} is $(x_1,\ldots,x_{a-1}) \in \bb{N}^{a-1}$, where $ax_i + i$ is the smallest element of the numerical semigroup congruent to $i$ mod $a$.
%By considering numerical sets with $a$ in their atom monoids,  we can extend the notion of the Ap\'ery tuple to these numerical sets.
%Finally, it is easily shown that these numerical sets correspond precisely to the $a$-core partitions.

We then use these correspondences to answer enumerative questions about hook sets of partitions.  For example, we give another proof of the following result of Berg and Vazirani.

%We use this correspondence to answer various questions about hook sets of partitions.
%For instance, we demonstrate that the size of an $a$-core partition is a quadratic function in the $x_i$, and we reprove a theorem of Berg and Vazirani \cite{Berg}:
\begin{proposition}[Proposition 3.1.4 in \cite{Berg}]\label{BergVazirani}
	The number of $a$-cores with $g$ parts is equal to the number of $(a-1)$-cores with less than or equal to $g$ parts.
\end{proposition}

%Consider some $b$ which is relatively prime to $a$.
%Requiring that any $a$-core partition is also a $b$-core is equivalent to requiring that the corresponding point of $\bb{N}^{a-1}$ satisfies certain linear inequalities.
%Since the number of $(a,b)$-cores is finite, this shows that the $(a,b)$-core partitions correspond with a rational polytope in $\bb{N}^{a-1}$.
%In fact, we prove the following more general statement.

%\begin{theorem}
%	If $\gcd(a,b_1,\ldots,b_m) = 1$, then the region of $\bb{N}^{a-1}$ corresponding with $(a,b_1,\ldots,b_m)$-core partitions is a rational polytope.
%\end{theorem}

Both Johnson and Chen, Huang, and Wang give proofs that the self-conjugate $(a,b)$-core partitions have the same average size as the set of all $(a,b)$-core partitions \cite{Chen, Johnson}.  It is natural to ask whether sizes of other subfamilies of $(a,b)$-cores have similar statistical properties.  We focus on two particular cases; $(a,b)$-cores that correspond to numerical semigroups under the map $\varphi$, and the set of all partitions with a given hook set.  Computational evidence suggests that the average size of an $(a,b)$-core corresponding to a numerical semigroup is not equal to the average size of all $(a,b)$-cores.  

In this setting we do not even have an analogue of Anderson's theorem on the number of these partitions.  This is equivalent to asking for the number of semigroups containing $a$ and $b$.  For a set of nonnegative integers $n_1,\ldots, n_t$ we define the numerical semigroup \emph{generated by} them to be 
\[
\langle n_1,\ldots, n_t \rangle = \left\{\sum_{i=1}^t a_i n_i\ |\ a_i \in \N\right\}.
\]
Note that any semigroup containing $a$ and $b$ also contains $\langle a,b\rangle$.  A semigroup $T$ containing a numerical semigroup $S$ is called an \emph{oversemigroup} of $S$.  Let $O(S)$ denote the number of oversemigroups of $S$.  Using a characterization due to Branco, Garc\'ia-Garc\'ia, Garc\'ia-S\'anchez,and Rosales for when an Ap\'ery tuple corresponds to a numerical semigroup we show that $O(\langle a,b\rangle)$ is equal to the number of lattice points in a certain rational polytope \cite{Branco}.  Hellus and Waldi have also studied this problem, giving formulas for small $a$ and bounds for the general case \cite{Hellus}.  We state their main result as Theorem \ref{HellusTheorem}.

%\begin{theorem}[Theorem 1.1 in \cite{Hellus}]
%	Let $O(a,b)$ be the number of semigroups containing $a$ and $b$.
%	Then for $a$ and $b$ coprime, $O(a,b)$ is a quasipolynomial in $b$ of degree $a-1$ with leading coefficient between $\frac{1}{(a-1)!a!}$ and $\frac{1}{(a-1)a!}$.
%\end{theorem}

%The finite set of simultaneous $(a,b)$-core partitions can be further divided by adding 
%ask for families of simultaneous $(a,b)$-core partitions who
%We also attempt to count the number of partitions which correspond to numerical semigroups under this bijection.
%For simultaneous $(a,b)$-core partitions, this turns out to be the number of semigroups containing $a$ and $b$.

%Hellus and Waldi in \cite{Hellus} calculate these values for small $a$, and give bounds for the general case.
%We give our own calculations for $a \leq 4$:
We give our own calculations for $a \le 4$ using using different methods.
\begin{theorem} \label{overs3}
	If $S = \langle 3,\ 6k+\ell \rangle$ with $\ell \in \{1,2,4,5\}$, then $O(S) = (3k+\ell)(k+1)$.
\end{theorem}

\begin{theorem} \label{overs4}
Suppose that $S = \langle 4, 12k + \ell\rangle$ with $\ell \in \{1,3,5,7,9,11\}$.  Then $O(S)$ is given by the following chart:

	\begin{center}
	\begin{tabular}{c | c}
	$\ell$	&	$O(S)$ \\
	\hline
	1	&	$24k^3 + 30k^2 + 11k + 1$ \\
	3	&	$24k^3 + 42k^2 + 23k + 4$ \\
	5	&	$24k^3 + 54k^2 + 39k + 9$ \\
	7	&	$24k^3 + 66k^2 + 59k + 17$ \\
	9	&	$24k^3 + 78k^2 + 83k + 29$ \\
	11	&	$24k^3 + 90k^2 + 111k + 45$
	\end{tabular}.
	\end{center}
	
	%If $S = \langle 4,\, 12k+1 \rangle$, then $O(S) = 24k^3 + 30k^2 + 11k + 1$.
\end{theorem}
%\begin{theorem}
%For a positive integer $k$,
%\[ 
%O(3,6k+\ell) = (3k+\ell)(k+1)
%\] 
%and 
%\[
%O(4,12k+1) = 24k^3 + 30k^2 + 11k + 1.
%\]
%\end{theorem}
%We can use similar techniques to give polynomial formulas like these for $O(3,b)$ and $O(4,b)$ for $b$ lying in different residue classes, but omit the details here.

%	The number of semigroups containing $\langle 3, 6k + \ell\rangle$ is $(3k+\ell)(k+1)$, and the number of semigroups containing $\langle 4, 12k + 1\rangle$ is $24k^3 + 30k^2 + 11k + 1$.
%	In general, the number of semigroups containing $\langle 4, 12k + \ell\rangle$ is on the order of $24k^3$.
%\end{theorem}

It is not difficult using the bijection $\varphi$ to show that the hook set of a partition is always the complement of a numerical semigroup.  Let $\N \sminus S$ denote the complement of some numerical semigroup $S$.  If $a$ and $b$ are not in $\N \sminus S$ then any partition with this hook set is a simultaneous $(a,b)$-core.  In order to study statistical questions about sizes of partitions with a given hook set, we would first like to understand how many partitions have this hook set.  We call this number $P(S)$.  We investigate how the properties of $S$ affect the behavior of this function, giving some results and suggesting questions for future work.  The \emph{Frobenius number} of a numerical set $T$ is the largest element of its complement and is denoted $F(T)$.  The size of the complement is called the \emph{genus} of $T$ and the elements of the complement are called the \emph{gaps} of $T$.  These concepts play important roles in our analysis of this problem.  

 The study of the set of partitions with a given hook set fits in nicely with previous work of Chung and Herman \cite{Chung}, and of Craven \cite{Craven}, on partitions with equal hook multisets.
%Another approach to the study of core partitions is to study when two partitions have the same hook set or hook multiset.
In \cite{Chung}, the authors show that a partition is uniquely determined up to reflection by its \emph{extended hook multiset}, in which hook lengths can take negative values. However, they also show that arbitrarily many distinct partitions can have the same hook multiset.  This result has been vastly generalized by Craven.

\begin{theorem}[Theorem 1.4 in \cite{Craven}]
Let $k$ and $\ell$ be natural numbers.  The for all sufficiently large $n$, there are $k$ disjoint sets of $\ell$ partitions of $n$, such that all of the $\ell$ partitions in each set have the same multiset of hook numbers, and distinct sets contain partitions with different hook numbers, and moreover different products of hook numbers.
\end{theorem}

Craven proves this result by defining certain classes of partitions that he calls \emph{enveloping partitions} that have the same hook multiset as many other partitions.  It is natural to ask what are the properties of the numerical semigroups giving the underlying hook sets of these partitions that make them suitable for this construction.  We focus on the opposite extreme.  A numerical set $T$ is called symmetric if for every $i \in [0,F(T)]$ exactly one of $i,F(T)-i$ is in $T$.  We prove that there is a unique partition with a given hook set if and only if that hook set is the complement of a symmetric numerical semigroup.  We also investigate the relationship between the function $P(S)$ and the number of \emph{missing pairs} of $S$, that is, the number of pairs $i, F(S)-i$ in the complement of $S$ with $i \in [0,F(S)/2]$.

We conclude the paper by discussing some asymptotic questions and conjectures based on computational evidence.

%In contrast, Craven in \cite{Craven} showed that for any integer $n$ there is a set of at least $n$ partitions which share the same hook multiset.

%Using the bijection of partitions with numerical sets, we study partitions with the same hook set, as opposed to hook multiset.
%In particular, we show that the hook set of any partition is the complement of the atom monoid of the associated numerical set.
%We then attempt to count the number of such partitions for a given atom monoid.
%We establish a subtle relationship of this number with the number of pairs $(a, F - a)$, where $F$ is the largest element of the complement of the atom monoid, and both $a$ and $F-a$ are in the complement as well.

%We conclude by noting some interesting asymptotic results with regards to partitions and numerical sets, and we discuss some conjectures based on current data.

% % % % % % % % % % % % % % % % % % % % % % % % % % % % % % %

\section{The correspondence between numerical sets and partitions}\label{SetsPartitions}

We first explain the bijection $\varphi$ introduced in the previous section connecting numerical sets to partitions and use it to find the relationship between atom monoids and hook sets.  We begin with an example.

%Our first step is to elaborate on the bijection between partitions and numerical sets, and use this to find the relationship between atom monoids and hook sets corresponding to partitions. A \emph{numerical set} is a subset of the nonnegative integers which contains $0$ and has a finite complement.
%If $T$ is a numerical set, the \emph{Frobenius number} of $T$ is the largest element in its complement, and is denoted $F(T)$.
%The $\emph{genus}$ of $T$ is the size of its complement, and is denoted $g(T)$.
%A numerical set which is also closed under addition is called a \emph{numerical semigroup} (despite being more specifically a monoid).

%The \emph{atom monoid} of a numerical set $T$ is
%	$$A(T) = \{n \in \bb{N} : n + T \subseteq T\}.$$
%Note that $A(T) \subseteq T$ since $0 \in T$, and $A(T) = T$ if and only if $T$ is a numerical semigroup. 
%The atom monoid is always closed under addition, so in some sense $A(T)$ is the underlying semigroup of $T$. Furthermore, $A(T)$ will have the same Frobenius number as $T$.

\begin{example}
	Let $T = \{0,1,4,5,7,\to\}$, where ``$\to$'' means that $T$ contains every integer greater than 7, as in the conventions of \cite{Rosales}.
	Clearly $T$ is a numerical set with $F(T) = 6$, $g(T) = 3$, and $A(T) = \{0,4,5,7,\to\}$.
\end{example}

Given a numerical set $T$ we construct a partition $\varphi(T)$ such that the map $\varphi$ is a bijection from numerical sets to partitions.
We construct $\varphi(T)$ by defining the profile of its Young diagram.  We can think of this path as lying in $\Z^2$ with the bottom left corner of the Young diagram at the origin. Starting with $n = 0$:
	\begin{itemize}
		\item if $n \in T$ draw a line of unit length to the right,
		\item if $n \notin T$ draw a line of unit length up,
		\item repeat for $n+1$.
	\end{itemize}
For any $n$ greater than the Frobenius number of $T$ we draw a line to the right.  As $T$ is a numerical set this process ends with an infinite set of steps to the right.  We disregard this section, forming the Young diagram with this profile walk, the line $x=0$, and this horizontal line.  The construction is understood most clearly with an example.

%so ultimately we can discount this section.
%The result is a staircase diagram lying along $\bb{Z}^2$ which can be considered as the rim of a Young diagram.
%This construction is easiest seen with an example.

\begin{example}
	If $T = \{0,1,4,5,7,\to\}$, then $\varphi(T) = (4,2,2)$:
	
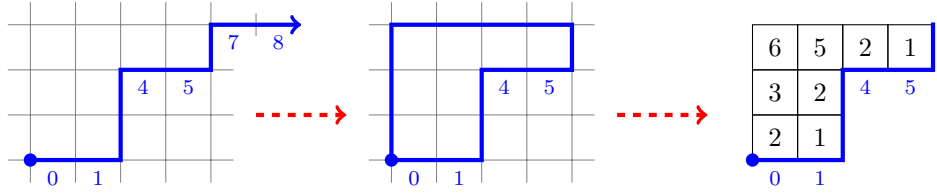
\begin{center}
		\begin{tikzpicture}[scale = 0.6]
		% T = \{0,1,4,5,7,\to\}
		% \lambda = (4, 2, 2)
		
			% 1 - rim lines
			\begin{scope}
				\draw[help lines] (-.5,-.5) grid (4.5,3.5);
				\draw[help lines] (5,2.75) -- (5,3.25);
				
				\begin{scope}[blue, ultra thick]
					\filldraw (0,0) circle [radius = 0.1];
					\draw[->, font = \footnotesize]
						(0,0) --
						node [below] {0} ++(1,0) --
						node [below] {1} ++(1,0) -- 
						node [right] {} ++(0,1) --
						node [right] {} ++(0,1) --
						node [below] {4} ++(1,0) --
						node [below] {5} ++(1,0) --
						node [right] {} ++(0,1) --
						node [below] {7} ++(1,0) --
						node [below] {8} ++(1,0);
				\end{scope}
			\end{scope}

			% 2 - rim with border
			\begin{scope}[xshift = 8cm]
				\draw[help lines] (-.5,-.5) grid (4.5,3.5);
				
				\begin{scope}[blue, ultra thick]
					\filldraw (0,0) circle [radius = 0.1];
					\draw[font = \footnotesize]
						(0,0) --
						node [below] {0} ++(1,0) --
						node [below] {1} ++(1,0) -- 
						node [right] {} ++(0,1) --
						node [right] {} ++(0,1) --
						node [below] {4} ++(1,0) --
						node [below] {5} ++(1,0) --
						node [right] {} ++(0,1) --
						(0,3) -- cycle;
				\end{scope}
			\end{scope}

			% 3 - rim with Young diagram and hooks
			\begin{scope}[xshift = 16cm]
				\foreach \x in {0,1,2,3}
					{ \draw (\x,2) rectangle ++(1,1);}
				\foreach \x in {0,1}
					{ \draw (\x,1) rectangle ++(1,1);}
				\foreach \x in {0,1}
					{ \draw (\x,0) rectangle ++(1,1);}
				
				\begin{scope}
					\node at (0.5,2.5) {6};
					\node at (1.5,2.5) {5};
					\node at (2.5,2.5) {2};
					\node at (3.5,2.5) {1};
					
					\node at (0.5,1.5) {3};
					\node at (1.5,1.5) {2};
					
					\node at (0.5,0.5) {2};
					\node at (1.5,0.5) {1};
				\end{scope}
				
				\begin{scope}[blue, ultra thick]
					\filldraw (0,0) circle [radius = 0.1];
					\draw[font = \footnotesize]
						(0,0) --
						node [below] {0} ++(1,0) --
						node [below] {1} ++(1,0) -- 
						node [right] {} ++(0,1) --
						node [right] {} ++(0,1) --
						node [below] {4} ++(1,0) --
						node [below] {5} ++(1,0) --
						node [right] {} ++(0,1.05);
				\end{scope}
			\end{scope}

			% arrows
			\begin{scope}[ultra thick, dashed,  red, ->]
				\draw (5,1) -- (7,1);
				\draw (13,1) -- (15,1);
			\end{scope}
		\end{tikzpicture}
	\end{center}
\end{example}

Seeing that $\varphi$ is a bijection is simple: to find the inverse image of a partition $\lambda$ label the profile of the Young diagram as above, starting with $0$.  The complement of the numerical set $\varphi^{-1}(\lambda)$ consists of the positive integers labeling the vertical steps of the profile.

We give some basic properties of $\varphi$ here, some of which might be evident from the example.  These results are clear from \cite{Keith} but we include them with proofs for completeness.

%There are a few properties of $\varphi$ worth mentioning now, some of which might be evident from the example.

\begin{proposition} \label{HooksAreDifferences}
	Given a numerical set $T$, the hook multiset of $\varphi(T)$ is 
		$$\cc H(\varphi(T)) = \{n - t : n \notin T, t \in T, n > t\}.$$
\end{proposition}

\begin{proof}
	Consider a box $B$ in the Young diagram of $\varphi(T)$ such that $B$ is in the same column as the horizontal step on the profile associated to $t \in T$ in the construction of $\varphi(T)$, and the same row the vertical step associated to $n \notin T$.
	
	Recall that the hook of $B$ is the set of boxes to the right (the ``arm''), the set of boxes below (the ``leg''), and $B$ itself. Counting steps along the profile shows that $n - t$ is the hook length of $B$.
\end{proof}

\begin{proposition} \label{HooksAreComplementOfAtom}
	Given a numerical set $T$, the hook set of $\varphi(T)$ is the complement of its atom monoid: $H(\varphi(T)) = \N \sminus A(T)$.
\end{proposition}

\begin{proof}
	By Proposition \ref{HooksAreDifferences} this amounts to proving that $\bb{N} \sminus A(T) = D$, where $D = \{n - t : n \notin T, t \in T, n > t\}$.

	Suppose $x \in \bb{N} \sminus A(T)$, so $x + t \notin T$ for some $t \in T$. This implies $(x + t) - t = x \in D$. Conversely, if $x \in D$ then $x = n - t$ for some $n \notin T$ and $t \in T$. This implies $x + t = n \notin T$, so $x \in \bb{N} \sminus A(T)$.
	
\end{proof}

\begin{remark}
	In particular, since $\varphi$ is bijective, Proposition \ref{HooksAreComplementOfAtom} shows that the hook set of any partition is the complement of a numerical semigroup.
	This implies that a partition is an $a$-core if and only if $a$ is not in its hook set, a simpler condition than having no hook lengths divisible by $a$.  
%This result is well-known, but it is still standard to give the definition of $a$-core from the introduction.
\end{remark}

% % % % % % % % % % % % % % % % % % % % % % % % % % % % % % % %

\section{The correspondence of $a$-cores and $\bb{N}^{a-1}$}\label{Correspondence}

In this section, we use the bijection between $a$-core partitions and $\N^{a-1}$ to prove several combinatorial results.  This correspondence comes from taking the Ap\'ery tuple of the numerical set associated to an $a$-core partition via the map $\varphi$.  By Proposition \ref{HooksAreComplementOfAtom}, a partition $\lambda$ is an $a$-core if and only if the atom monoid of $\varphi^{-1}(\lambda)$ contains $a$. 
%In this case we have defined the Ap\'ery set and Ap\'ery tuple of the corresponding numerical set in the introduction.  
For $a \in A(T)$ we have $n + a \in T$ for any $n \in T$, which means that if $x_i \in \Ap(T),\ x_i + ka \in T$ for any $k \in \N$. This shows that the Ap\'ery set and Ap\'ery tuple uniquely determine a numerical set whose atom monoid contains $a$. 

%prove that $a$-core partitions are in correspondence to the Ap\'ery tuples of $\bb{N}^{a-1}$ and further examine properties of $a$-cores.  

%\begin{definition}
%	Given a numerical set $T$ such that $a \in A(T)$, the \emph{Ap\'ery set of $T$ with respect to $a$} is
%		$$\Ap(T,a) = \{n \in T : n - a \notin T\},$$
%	or equivalently $\Ap(T,a) = \{0,w_1, \ldots, w_{a-1}\}$ where $w_i$ is the smallest element of $T$ congruent to $i$ mod $a$ for each $1 \leq i \leq a - 1$.
%	If $w_i = ax_i + \ell_i$, we define the \emph{Ap\'ery tuple} of $T$ with respect to $a$ (also called the \emph{Kunz-coordinate vector}) to be the %$a-1$ tuple 
%		$$\Ap'(T,a) = (x_1, \ldots, x_{a-1}) \in \bb{N}^{a-1}.$$
%	When there is no ambiguity we will omit the phrase ``with respect to $a%$.''
%\end{definition}

%Note that we may just as well have defined the Ap\'ery set and tuple for any numerical set, but it would not necessarily hold any meaning.

%It is also worth mentioning that the Ap\'ery tuple defined here is the same concept as the {Kunz-coordinates vector} for numerical semigroups defined by Blanco and Puerto in \cite{Blanco}.
%Here we simply generalized the idea to certain numerical sets.

Consider $(x_1, \ldots, x_{a-1}) \in \N^{a-1}$ and the associated numerical set $T = \{ ax_i + i + ma : m \in \bb{N}, 1 \leq i \leq a-1\}$. We see that $a \in A(T)$ and $\Ap(T) = (x_1, \ldots, x_{a-1})$.
Hence $\N^{a-1}$ is in bijection with numerical sets whose atom monoid contains $a$.  

In summary, we have the following one-to-one correspondences:
	$$		\left\{\begin{array}{c}
				\text{$a$-core} \\
				\text{partitions}
			\end{array}\right\}
		\longleftrightarrow \left\{
			\begin{array}{c}
				\text{numerical sets} \\
				\text{whose atom monoid} \\
				\text{contains $a$}
			\end{array}\right\}
		\longleftrightarrow \left\{
			\begin{array}{c}
				\text{the tuples} \\
				\text{of $\bb{N}^{a-1}$}
			\end{array}\right\},$$
completing the proof of Proposition \ref{bijections}. Note that the origin of $\N^{a-1}$ corresponds with the numerical set $T = \bb{N}$, which corresponds with the empty partition, an $a$-core for any $a$.

Recall that the Frobenius number $F(T)$ of the numerical set $T$ is the maximum element of its complement. Note that $F(T) \notin A(T)$ but $n \in A(T)$ for any $n > F(T)$.  By Proposition \ref{HooksAreComplementOfAtom}, the maximum hook length of $\varphi(T)$ is $F(T)$.  Also, if $a \in A(T)$ and $\Ap(T)= (x_1,\ldots,x_{a-1})$, then $F(T) = \max\{ax_i + i - a\}$.
The above bijection allows us to easily compute the number of $a$-core partitions by maximum hook length.

\begin{proposition}
	For $1 \leq \ell \leq a-1$, the number of $a$-core partitions with maximum hook length $ak + \ell$ is $(k+2)^{\ell-1}(k+1)^{a-\ell-1}$.
\end{proposition}

\begin{proof}
	The $a$-core partitions with maximum hook length $ak + \ell$ are those for which $\max\{ax_i + i - a\} = ax_\ell + \ell - a$ where $x_\ell = k+1$.
	This implies
		$x_\ell > x_i$ for any $i > \ell$, and
		$x_\ell \geq x_i$ for any $i < \ell$.
	Therefore, such partitions are in bijection with choices for the $x_i$ satisfying $x_i \in [0,k]$ for any $i \in [\ell+1,a-1]$ and $x_i \in [0,k+1]$ for any $i \in [1,\ell-1]$.
	
\end{proof}

\begin{proposition} \label{aCoresMaxHookLessThanAK}
	For any $k \in \bb{N}$, the number of $a$-core partitions with maximum hook length less than $ak$ is $(k+1)^{a-1}$.
\end{proposition}

\begin{proof}
	An $a$-core partition $\lambda$ has maximum hook length less than $ak$ if and only if $\max\{ax_i + i - a\} < ak$, where $(x_1,\ldots,x_{a-1})$ is the Ap\'ery tuple of the corresponding numerical set. This holds if and only if $x_i \leq k$ for each $i$.
	Therefore the $a$-core partitions with maximum hook length less than $ak$ are those which correspond with the lattice points of $[0,k]^{a-1} \subset \N^{a-1}$.
\end{proof}

We can similarly find the number of $a$-cores with a fixed number of parts.  The construction of $\varphi(T)$ from $T$ shows that the number of parts is equal to the size of $\N \sminus T$.  So the number of parts of $\varphi(T)$ is equal to $g(T)$.

\begin{proposition}\label{bergProp1}
	The number of $a$-core partitions with $g$ parts is $\binom{g+a-2}{a-2}$.
\end{proposition}

\begin{proof}
Suppose $a \in A(T)$ and $\Ap(T) = (x_1,\ldots,x_{a-1})$. 
If $n \equiv i \Mod a$ then $n \in T$ if and only if $n \geq ax_i + i$. Therefore the genus of $T$ is $x_1 + \cdots + x_{a-1}$, and the number of $a$-cores with $g$ parts is equal to the number of points of the simplex $(x_1,\ldots,x_{a-1}) \in \N^{a-1}$ such that $x_1 + \cdots + x_{a-1} = g$.
It is well-known that there are $\binom{g+a-2}{a-2}$ such points.
\end{proof}

\begin{proposition}\label{bergProp2}
	The number of $a$-core partitions with less than or equal to $g$ parts is $\binom{g+a-1}{a-1}$.
\end{proposition}

\begin{proof}
The number of numerical sets $T$ with $a \in A(T)$ and genus less than or equal to $g$ is the number of points $(x_1,\ldots,x_{a-1}) \in \bb{N}^{a-1}$ such that $x_1 + \cdots + x_{a-1} \leq g$.
Counting these points is the same as counting the number of points $(x_1,\ldots,x_{a-1},y) \in \bb{N}^{a}$ such that $x_1+\ldots+x_{a-1} + y = g$. Therefore there are $\binom{g+a-1}{a-1}$ such points.
\end{proof}

These two results together give another proof Berg and Vazirani's Proposition \ref{BergVazirani} stated in the introduction \cite{Berg}.

Since the conjugate of an $a$-core partition is also an $a$-core, the number of $a$-cores with $g$ parts is equal to the number of $a$-cores with largest part $g$.
Hence Propositions \ref{bergProp1}, \ref{bergProp2}, and \ref{BergVazirani} may be restated with ``largest part $g$'' in place of ``$g$ parts''.

We close this section by giving another interpretation of Theorem 1.9 of \cite{Johnson} where Johnson relates the size of a partition corresponding to a quadratic function evaluated at the associated lattice point.  Since our correspondence between core partitions and lattice points is different we get a different function, but the ideas are similar.

\begin{proposition}
Let $T$ be a numerical set with $a \in A(T)$ and $Ap(T) = (x_1,\ldots, x_{a-1})$.  Then the size of the partition $\varphi(T)$ is
\begin{eqnarray*}
F_a(x_1, \ldots, x_{a-1})  & = &  \frac{a}{2} \sum_{i=1}^{a-1} x_i(x_i - 1) + \sum_{i=1}^{a-1} i x_i - \frac{1}{2}\left(\sum_{i=1}^{a-1} x_i\right)\left(-1 + \sum_{i=1}^{a-1} x_i\right)  \\
& = &  \frac{a-1}{2} \sum_{i=1}^{a-1} x_i^2 + \sum_{i=1}^{a-1} \left(i-\frac{a-1}{2}\right) x_i  - \sum_{1 \le i < j \le a-1} x_i x_j
\end{eqnarray*}
\end{proposition}

\begin{proof}
In the proof of Proposition \ref{bergProp2} we noted that the genus of a numerical set $T$ is the sum of the elements of the corresponding Ap\'ery tuple.  As noted above, the number of parts of $\varphi(T)$, which is equal to the number of rows of its Young diagram, is given by the genus of $T$.  By Proposition \ref{HooksAreDifferences} the hooks in the first column of the Young diagram are exactly the elements of $\N \sminus A(T)$.
%This also means that the set of hooks in the leftmost column of $\varphi(T)$ are exactly the gaps of $T$.
By the definition of a hook, the sum of these hook lengths is almost the size of $\varphi(T)$, except that we have overcounted the $i$-th box from the top $i-1$ times. This means we have overcounted $(g(T) - 1)g(T)/2$ boxes in the Young diagram and the size of $\varphi(T)$ is the sum of the gaps of $T$ minus $(g(T) - 1)g(T)/2$.

%Another point worth making about this correspondence is that there is a quadratic function $F_a:\bb{N}^{a-1} \to \bb{N}$ such that $F_a(p)$ is the size of the $a$-core partition corresponding to $p \in \bb{N}^{a-1}$.

%The first step in constructing $F_a$ is to notice that the genus of $T$ is the sum of elements in its Ap\'ery tuple.
%To see this, suppose $a \in A(T)$ and $\Ap'(T,a) = (x_1,\ldots,x_{a-1})$.
%As mentioned before, this implies that $ax_i + i$ is the smallest element of $T$ congruent to $i \Mod a$.
%There are $x_i$ positive integers less than $ax_i + i$ that are congruent to $i$, none of which are in $T$, which accounts for all elements of $\bb{N}\sminus T$ congruent to $i$.
%Since any element of the complement of $T$ (also called the \emph{gaps} of $T$) is congruent to some $1 \leq i \leq a-1$ we have that $g(T) = \sum_{i=1}^{a-1} x_i$.

%In the construction of $\varphi(T)$ we constructed a vertical line (creating a new part in the partition) when the side would be labeled with a gap of $T$, so the number of parts of $\varphi(T)$---equal to the height of its Young diagram---is $g(T)$.

If $ax_i + i$ is the smallest element of $T$ congruent to $i \Mod a$, then the sum of gaps congruent to $i$ is
	$$\sum_{n=0}^{x_i-1} an + i = \frac{a x_i (x_i-1)}{2} + ix_i.$$
Summing over all $i \in [1,a-1]$ and using $g(T) = \sum_{i=1}^{a-1} x_i$ completes the proof. 
%discussed above
%we come to the function
%	$$F_a(x_1, \ldots, x_{a-1}) = \frac{a}{2} \sum_{i=1}^{a-1} x_i(x_i - 1) + \sum_{i=1}^{a-1} i x_i - \frac{1}{2}\left(\sum_{i=1}^{a-1} x_i\right)\left(-1 + \sum_{i=1}^{a-1} x_i\right),$$
%which gives the size of the partition corresponding to $(x_1,\ldots, x_{a-1})$.
\end{proof}

% % % % % % % % % % % % % % % % % % % % % % % % % % % % % % % % 

\section{The $(a,b)$-core Polytope}\label{abCorePolytope}

In this section we use the bijections of Proposition \ref{bijections} to prove  Theorem \ref{KeithNath} and the stronger result that simultaneous $(a,b_1,\ldots, b_m)$-core are in bijection with lattice points of a polytope that we define below.  For now, we do not necessarily assume that $\gcd(a,b) = 1$ but we do assume that $a\nmid b$.  Suppose that $b = ak + \ell$ where $\ell \in [1,a-1]$ and that $T$ is a numerical set such that $\varphi(T)$ is an $a$-core partition with Ap\'ery tuple $\Ap(T) = (x_1, \ldots, x_{a-1})$. By the remark following Proposition \ref{HooksAreComplementOfAtom}, $\varphi(T)$ is a $b$-core partition if and only if $b \in A(T)$, which is true if and only if $a x_i + i + b \in T$ for all $i \in [1,a-1]$.

%With these correspondences in mind we are ready to show that for any $b \in \bb{N}$, the region of $\bb{N}^{a-1}$ corresponding to $b$-cores is defined by a set of rational inequalities.
%In particular, if $\gcd(a,b) = 1$ we show that the region of $\bb{N}^{a-1}$ corresponding to $(a,b)$-core partitions is a rational polytope.

%Previously we have discussed $b$ as an integer prime to $a$, but for the moment we will not make this assumption.
%Suppose $b = ak + \ell$, where $\ell < a$, and suppose $T$ is a numerical set such that $\varphi(T)$ is an $a$-core partition with Ap\'ery tuple $\Ap'(T,a) = (x_1, \ldots, x_{a-1})$.

If $i + \ell < a$ then $a x_i + i + b \in T$ if and only if $a x_i + i + b \geq a x_{i + \ell} + (i + \ell)$. Similarly if $i + \ell > a$ then $a x_i + i + b \in T$ if and only if $a x_i + i + b \geq a x_{i + \ell - a} + (i + \ell - a)$.
Therefore $\varphi(T)$ is a $b$-core if and only if $Ap(T)$ satisfies the inequalities
	\begin{align*}
		x_\ell & \leq  k, &  \\
		x_{i + \ell} & \leq k + x_i,  	& \text{ if } i + \ell < a, \\
		x_{i + \ell - a} & \leq k + x_i + 1, 	& \text{ if } i + \ell > a,\\
		x_i & \geq 0. 
	\end{align*}
Let $\cc P_{a,b} \subseteq \R^{a-1}$ be the region defined by the intersection of these half-spaces.  This is a rational polyhedral cone and is a rational polytope if and only if it is bounded, which is true if and only if $\gcd(a,b) = 1$.  We now state and prove a more general result.

%If we now assume that $a$ and $b$ are coprime, to show $\cc P_{a,b}$ is a rational polytope we only need to show that the region they define is bounded.
%This is not hard to see since $a$ and $\ell$ must be coprime, but we will prove a more general result instead.

\begin{proposition} \label{multicoreBijection}
	Suppose $\gcd(a,b_1, \ldots, b_m) = 1$ where we write $b_j = ak_j + \ell_j$ for each $j \in [1,m]$ with $\ell_j \in [1,a-1]$.  There is a bijection between $(a,b_1,\ldots,b_m)$-core partitions and the integer points of the polytope defined by the following inequalities: %taken over all $1 \leq j \leq m$:
		\begin{align}
			x_{\ell_j} &\leq k_j, \label{Pab1} \\
			x_{i + \ell_j} &\leq k_j + x_i, 	& \text{ if } i + \ell_j < a, \label{Pab2} \\
			x_{i + \ell_j - a} &\leq k_j + x_i + 1, 	& \text{ if } i + \ell_j > a, \label{Pab3} \\
			x_i &\geq 0 & \text{for } i \in [1,a-1], \label{Pab4}
		\end{align}
where we have one set of inequalities (\ref{Pab1}), (\ref{Pab2}), and (\ref{Pab3}) for each $j \in [1,m]$. 

	%Then the region of $\bb{N}^{a-1}$ corresponding with $(a,b_1,\ldots,b_m)$-core partitions is a rational polytope.
	%In other words, if we write $b_j = ak_j + \ell_j$ for each $1 \leq j \leq m$, then there is a bijection between 
	\end{proposition}

\begin{proof}
	Let $\cc Q$ be the intersection of the half-spaces defined by these inequalities and note that $\cc Q = \bigcap_{j = 1}^m \cc P_{a,b_j}$.  The lattice points of this region are in bijection with $(a,b_1,\ldots,b_m)$-cores, so we need only show that $\cc Q$ is bounded.  Suppose $(x_1,\ldots, x_{a-1}) \in \cc Q$ with each $x_i$ a nonnegative integer.  We give an upper bound on each $x_i$ that depends only on $a, b_1,\ldots, b_m$, which completes the proof.

%	If $(x_1,\ldots,x_{a-1}) \in Q$ then we need to show that each $x_i$ is bounded.
	Since $(x_1,\ldots,x_{a-1})$ satisfies (\ref{Pab1}) for each $j\in [1,m]$, we see $x_{\ell_j} \leq k_j$.  After reindexing, (\ref{Pab2}) implies
		$x_i \leq k_j + x_{i - \ell_j}$ if $i > \ell_j,$
and (\ref{Pab3}) implies
		$x_i \leq k_j + 1 + x_{i - \ell_j + a}$ if $i < \ell_j.$
	Hence
		$$x_i \leq k_j + 1 + x_{i - \ell_j \pmod* a},$$
	for each $j \in [1,m]$ where we write $x_{i \pmod* a}$ as shorthand for $x_{i'}$ where $i' \in [1,a-1]$ and $i' \equiv i \Mod a$.  Therefore
		$$x_{\ell_{j_1} + \ell_{j_2} \Mod* a} \leq k_{j_1} + 1 + x_{\ell_{j_2}} \leq k_{j_1} + k_{j_2} + 2.$$
	Proceeding by induction, if $s = \sum_{j=1}^m y_j \ell_j$ for some $y_1,\ldots, y_m \in \bb{N}$ then
		$$x_{s \pmod* a} \leq \sum_{j=1}^m y_j(k_j + 1).$$
		
Since $\gcd(a,b_1,\ldots, b_m) = 1$, for each $i \in [1,a-1]$ there exist nonnegative integers $y_1,\ldots, y_m$ such that $\sum_{i=1}^m y_j \ell_j \equiv i \pmod{a}$.  This gives an upper bound on each $x_i$ depending only on $a,b_1,\ldots, b_m$.
	
% By the assumption that $\gcd(a,b_1,\ldots,b_m) = 1$ we have $wa + \sum_{j=1}^m z_jb_j = 1$ for some $w,z_1,\ldots,z_m \in \bb{Z}$.
%	This implies $\sum_{j=1}^m z_j \ell_j \equiv 1 \Mod a$.
%	Choose $\lambda_j \in \bb{N}$ such that $z_j + \lambda_j a > 0$ for each $j$.
%	Then $\sum_{j=1}^m (z_j + \lambda_j a)\ell_j \equiv 1\Mod a$, and therefore for each $1 \leq i \leq a-1$,
%		$$\sum_{j=1}^m i (z_j + \lambda_j a) \ell_j \equiv i \!\Mod a.$$
%	Fixing $i$ and setting $y_j = i(z_j + \lambda_j a) \in \bb{N}$ and $s = \sum_{j=1}^m y_j \ell_j$ we have $s \equiv i \Mod a$, and so
%		$$x_i = x_{s \!\!\Mod a} \leq \sum_{j=1}^m y_j(k_j + 1),$$
%	which is finite.
%	Hence $\cc Q$ is bounded.
	
\end{proof}

In particular, this proves Theorem \ref{KeithNath}.  A formula for the number of integer points of the polytope $\cc Q$ is equivalent to a formula for the number of $(a,b_1,\ldots, b_m)$-cores.  For example, giving such a formula in the $m=1$ case is equivalent to Theorem \ref{AndersonsTheorem} of Anderson.  We note that several results in this area can be phrased in terms of counting integer points in special polytopes \cite{Amdeberhan, Xiong1, Yang}.

In general it is difficult to give a formula for the number of integer points of a polytope in terms of the defining half-spaces but there are some circumstances in which the polytopes are particularly nice.  For example, we can use this method to give another proof of Proposition \ref{aCoresMaxHookLessThanAK}.

%Counting the integer lattice points of the polytope $\cc Q$ would give the total number of $(a,b_1,\ldots,b_m)$-cores.
%In fact, finding this count for the case $m = 1$ would be a new proof of Anderson's result on the total number of $(a,b)$-cores (Theorem \ref{AndersonsTheorem}).

%In general, counting the integer points in polytopes is not such an easy task, but the polytopes corresponding to $(a,b_1,\ldots,b_m)$-cores are relatively ``nice'' in some circumstances.

\begin{proof}[Second proof of Proposition \ref{aCoresMaxHookLessThanAK}]
	The set of $a$-core partitions with maximum hook length less than $ak$ is exactly the set of $(a,\, ak+1,\ldots,\, ak+(a-1))$-core partitions, since an $(a,b)$-core is also an $(a+b)$-core by Proposition \ref{HooksAreComplementOfAtom}.
	This set corresponds with the lattice points of the polytope $\cc Q = \bigcap_{i=1}^{a-1} \cc P_{a,\, ak+i}$.
	By (\ref{Pab1}) we have $\cc Q \subseteq [0,k]^{a-1}$, and by (\ref{Pab2}) and (\ref{Pab3}) we have $[0,k]^{a-1} \subseteq \cc P_{a,ak+i}$ for each $i$.
	Therefore $\cc Q = [0,k]^{a-1}$, which contains $(k+1)^{a-1}$ integer points, so there are $(k+1)^{a-1}$ $a$-core partitions with maximum hook length less than $ak$.

\end{proof}

We close this section with a suggestion for future research.  Formulas for the number of integer points in families of rational polytopes can be be quite subtle, particularly when the polytope has vertices with large denominators.  The volume of a polytope is often a good approximation for its number of integer points and is usually easier to find.
\begin{problem}\label{Prob1}
Give an approximation for the volume of the $(a,b_1,\ldots, b_m)$-core polytope in terms of the integers $a,b_1,\ldots, b_m$.
\end{problem}

% % % % % % % % % % % % % % % % % % % % % % % % % % % % % % % %

\section{Counting $(a,b)$-cores from semigroups}

In this section we further investigate the correspondence between numerical sets with atom monoid containing $a$ and $a$-core partitions.  We focus on a natural subclass of these numerical sets, those that are actually numerical semigroups.  Recall that a numerical set is a numerical semigroup if and only if it is closed under addition, or equivalently, it is equal to its atom monoid.  We see that the bijection $\varphi$ takes a numerical semigroup $S$ to an $a$-core partition if and only if $a \in S$.  Our main goal in this section is to describe the set of $a$-core partitions that come from numerical semigroups and to count the set of simultaneous $(a,b)$-cores that come from semigroups for certain pairs $(a,b)$.

Recall from Theorem \ref{AndersonsTheorem} that for positive integers $a,b \ge 2$ with $\gcd(a,b) = 1$ the total number of $(a,b)$-cores is 
\[
C(a,b) = \frac{1}{a+b}\binom{a+b}{a}.
\]  
We are interested in finding the proportion of these partitions which come from semigroups via the map $\varphi$.  To do this we first show that these partitions are in bijection with the lattice points of a polytope contained in the $(a,b)$-core polytope of the previous section.  

A direct consequence of Proposition \ref{HooksAreComplementOfAtom} is that for a numerical semigroup $S$ the partition $\varphi(S)$ is an $(a,b)$-core if and only if $a,b \in S$.  Since $S$ is a semigroup it must also contain $\langle a,b\rangle$.  Our goal is to give formulas for $O(\langle a,b \rangle)$ in terms of $a$ and $b$ and to investigate the ratio $O(\langle a,b\rangle)/C(a,b)$.

Hellus and Waldi have studied exactly this problem in \cite{Hellus}.  They show that the set of oversemigroups of $\langle a,b\rangle$ are naturally in bijection with the set of integer points in a rational polytope.  For $a$ fixed and $b$ increasing they show that computing $O(\langle a,b\rangle)$ is equivalent to counting lattice points in dilates of this polytope and that they can therefore use techniques from Ehrhart theory to study the behavior of $O(\langle a,b \rangle)$.  This is notable because Ehrhart theory is also a major input of Johnson's proof of Armstrong's conjecture \cite{Johnson}.  In particular, they prove the following result.  A \emph{quasipolynomial} of degree $d$ is a function $f : \N^d \rightarrow \C$ of the form
\[
f(n) = c_d(n) n^d + c_{d-1}(n) n^{d-1} + \cdots + c_0(n)
\]
with periodic functions $c_i$ having integer periods, $c_d \neq 0$.
\begin{theorem}[Theorem 1.1 in \cite{Hellus}]\label{HellusTheorem}
Let $a \in \N,\ a > 1$.
\begin{enumerate}
\item There is a quasipolynomial of degree $a-1$ taking the value $O(\langle a,b \rangle)$ at each $b \in \N$ relatively prime to $a$.
\item The leading coefficient $c_{a-1}(n)$ of this quasipolynomial is constant and satisfies
\[
\frac{1}{(a-1)!\cdot a!} \le c_{a-1}(n) \le \frac{1}{(a-1) \cdot a!}.
\]
\item The function $O(\langle a,b \rangle)$ is increasing in both variables.
\end{enumerate}
\end{theorem}
Hellus and Waldi note that the upper and lower bounds of the second part of the statement coincide for $a = 2,3$, that the upper bound is correct for $a=4$, and that for $a=5,6,7$ the correct value lies strictly between the upper and lower bound \cite{Kunz}.  With the above theorem, finding the quasipolynomial $O(a,b)$ for fixed $a$ can be done with a finite amount of computation.  We also note that the idea of using Ehrhart theory to give quasipolynomial formulas for quantities associated to numerical semigroups also appears in \cite{Kaplan}.

We give our own calculations for $a \le 4$, showing how to derive formulas of this type without prior knowledge that the answer is given by a quasipolynomial. We explicitly describe the $a-1$ dimensional polytope whose integer points are in bijection with the oversemigroups of $\langle a,b\rangle$ and then divide this into $a-2$ dimensional slices via parallel hyperplanes.  We find exact formulas for the number of integer points in each slice.

%It is important to note here that Hellus and Waldi have studied exactly this problem in \cite{Hellus}.
%They showed that $O(\langle a, b \rangle)$ is a \emph{quasipolynomial} in $b$ of degree $a-1$, meaning $O(\langle a, b \rangle)$ is a polynomial in which the coefficients change depending on the residue class of $b$ mod $p$ for some period $p \in \bb{N}$.
%In fact, they show
%	\begin{theorem}[Hellus and Waldi \cite{Hellus}]
%		\label{HellusTheorem}
%		The leading coefficient of $O(\langle a, b \rangle)$ is constant, and between $\frac{1}{(a-1)!a!}$ and $\frac{1}{(a-1)a!}$.
%	\end{theorem}
%Hellus and Waldi also give calculations for $O(\langle a, b \rangle)$ for small values of $a$.

%In the discussion that follows we will reprove these calculations for $a \leq 4$ by using the bijection of partitions with numerical sets.
%We demonstrate first that these partitions are in bijection with a certain polytope of $\bb{N}^{a-1}$.
%Then using a specific series of cuts, we divide this $a-1$ dimensional polytope into many $a-2$ dimensional polytopes.
%Counting the integer points of the smaller polytopes is a much easier task, allowing us to sum over the smaller polytopes to find the total number of points (and hence partitions) in which we are interested.

Our first goal is to give defining inequalities for the polytope whose integer points are in bijection with oversemigroups of $\langle a,b\rangle$.  This is equivalent to determining when an Ap\'ery tuple $(x_1,\ldots, x_{a-1})$ of a numerical set containing $a$ actually corresponds to a numerical semigroup containing $a$.  The following result of Branco, Garc\'ia-Garc\'ia, Garc\'ia-S\'anchez, and Rosales, a slight variation of Theorem 11 in \cite{Branco}, gives this characterization.

\begin{theorem}[Theorem 11 in \cite{Rosales}]
The map from a numerical semigroup to its Ap\'{e}ry tuple gives
a one-to-one correspondence taking semigroups $T$ containing $a$ to solutions $(k_1,\ldots, k_{a-1})$
of the system of inequalities
\begin{eqnarray}
x_i  \in \N & & \text{for all}\ i\in \{1,\ldots, a-1\} \\
x_i + x_j \ge x_{i+j}  & &\text{for all}\ 1 \le i \le j \le a-1,\
i+j \le a-1 \label{ineq-semi1}\\
x_i +x_j + 1 \ge x_{i+j-a} & &\text{for all}\ 1 \le i \le j \le a-1,\
i+j > a.  \label{ineq-semi2}
\end{eqnarray}
\end{theorem}

%First, we consider the conditions necessary to be an oversemigroup of some given semigroup in terms of Ap\'ery tuples.
%Let $a$ and $b$ be coprime, let $S = \langle a, b \rangle$, and suppose $\Ap'(S,a) = (k_1, \ldots, k_{a-1})$.
%Suppose $T$ is a numerical set with $a \in A(T)$ and $\Ap'(T,a) = (\ell_1, \ldots, \ell_{a-1})$.
%As noted by Rosales et. al. in \cite{Branco}, $T$ is a numerical semigroup if and only if the following inequalities are satisfied:
%	\begin{IEEEeqnarray}{rCll}
%	\ell_i + \ell_j &\geq& \ell_{i+j} & \qquad \text{ if }i+j < a \label{ineq-semi1}\\
%	\ell_i + \ell_j +1 &\geq& \ell_{i+j-a} & \qquad \text{ if } i+j > a  \label{ineq-semi2}
%	\end{IEEEeqnarray}
Also, notice that $T \supseteq S$ if and only if
	\begin{equation}
	\ell_i \leq k_i \qquad \text{ for all $i$}.\label{ineq-over1}
	\end{equation}
Therefore, the set of inequalities (\ref{ineq-semi1}) -- (\ref{ineq-over1}) give necessary and sufficient conditions for $T$ to be an oversemigroup of $S$.

These inequalities define an $a-1$ dimensional polytope in which the lattice points correspond exactly with the oversemigroups of $S$.
In order to count the number of oversemigroups of $S$ we only need to count these lattice points. This polytope is of course contained in $\cc P_{a,b}$.  Hellus and Waldi study a similar polytope, but phrase their results in terms of counting lattice paths and do not make a connection to general $(a,b)$-core partitions or numerical sets \cite{Hellus}.

\begin{example}
	Consider $S = \langle 3, 8 \rangle$.
	The inequalities (\ref{ineq-semi1}) -- (\ref{ineq-over1}) reduce to
	\begin{align*}
	2x &\geq y \\
	2y + 1 &\geq x\\
	x &\leq 5 \\
	y &\leq 2
	\end{align*}
	which define the polytope:
	\begin{center}
		\tikzstyle{loosely dashed}=          [dash pattern=on 6pt off 6pt]
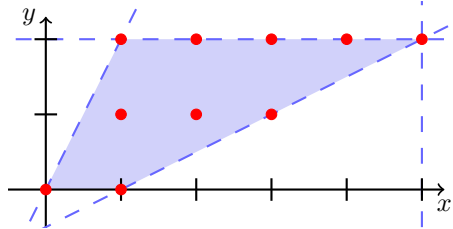
		\begin{tikzpicture}
			[point/.style={red, radius=0.07},
			ineq/.style={loosely dashed, blue!60, thick},
			scale=1]
	
			\clip (-0.5,-0.5) rectangle (5.5,2.5);
	
			\fill [gray!20!blue!20] (0,0) -- (1,0) -- (5,2) -- (1,2) -- cycle;
			%\draw [help lines, dashed] (-0.8,-0.9) grid (2.8,4.9);
			\draw [thick, ->] (-0.9,0) -- (5.3,0)
				node [below] {$x$};
			\draw [thick, ->] (0,-0.9) -- (0,2.3)
				node [left] {$y$};
			\draw [ineq, domain=-.4:1.2] plot (\x, {2*\x});
			\draw [ineq, domain=-.5:5.5] plot (\x, {\x/2 - 1/2});
			\draw [ineq, domain=-.4:5.5] plot (\x, {2});
			\draw [ineq] (5,-.5) -- (5,4.5);
	
			\foreach \x in {1,2,3,4,5} {\draw [thick] (\x,-0.15) -- (\x,0.15);}
			\foreach \y in {1,2,3,4} {\draw [thick] (-0.15,\y) -- (0.15,\y);}
	
			\filldraw [point] (0,0) circle;
			\filldraw [point] (1,0) circle;
			\filldraw [point] (1,1) circle;
			\filldraw [point] (1,2) circle;
			\filldraw [point] (2,1) circle;
			\filldraw [point] (2,2) circle;
			\filldraw [point] (3,1) circle;
			\filldraw [point] (3,2) circle;
			\filldraw [point] (4,2) circle;
			\filldraw [point] (5,2) circle;
		\end{tikzpicture}
	\end{center}
	Each lattice point $(x,y)$ in this polytope uniquely corresponds to an oversemigroup of $S$, and thus with a $(3,8)$-core partition.
	There are 10 integer points in this polytope, so $O(3, 8) = 10$ and there are $10$ simultaneous $(3,8)$-core partitions associated to numerical semigroups.
\end{example}

\vspace{5mm}

It seems difficult to give a general formula for $O(\langle a,b \rangle)$ so we begin by analyzing the cases $a = 2,3,4$, finding explicit formulas for each.  When $a = 2$ it is clear that $O(\langle 2, 2k+1 \rangle) = k+1$ since any oversemigroup of $\langle 2, 2k+1 \rangle$ is determined uniquely by its smallest odd element.  Our next goal is to prove Theorem \ref{overs3} and Theorem \ref{overs4} that were stated in the introduction.  We note that both results express $O(\langle a,b\rangle)$ as a quasipolynomial in $b$ of degree $a-1$, and that they agree with the calculations in \cite{Hellus}.

%A general formula for $O(\langle a, b \rangle)$ turns out to be difficult to derive, so we focus on the cases $2 \leq a \leq 4$, finding explicit formulas for each.
%When $a = 2$, it is clear that $O(\langle 2, 2k+1 \rangle) = k+1$ since any oversemigroup of $\langle 2, 2k+1 \rangle$ is determined uniquely by its smallest odd element, so we focus mainly on $a = 3$ and $a = 4$.

%\begin{theorem} \label{overs3}
%	If $S = \langle 3,\ 6k+\ell \rangle$ with $\ell \in \{1,2,4,5\}$, then $O(S) = (3k+\ell)(k+1)$.
%\end{theorem}
%This is indeed a quasipolynomial in $b$ of degree $a-1$, and agrees with the calculations in \cite{Hellus}.

In order to prove Theorem \ref{overs3}, we divide up the set of oversemigroups of $S = \langle 3,6k + \ell \rangle$ by genus.  It is easy to show that the genus of $\langle a,b \rangle$ is $(a-1)(b-1)/2$ \cite{Rosales}. For each integer $n \in [0,6k+\ell-1]$, the genus of $S$ we compute the number $O_n(S)$ of oversemigroups of $S$ with genus $n$.

%The genus of such a semigroup is between $0$ and the genus of $\langle 3,6k+\ell\rangle$, which is well-known to be $\frac{2(6k +\ell -1)}

%the problem by genus.
%The number of oversemigroups of $S$ with some fixed genus is a nice monotonic function which is more easily calculated.
%Let $O_n(S) = \#\{T \supseteq S : g(T) = n,\ T \text{ is a semigroup}\}$, the number of oversemigroups of $S$ with genus $n$.

\begin{lem} \label{overs3byG}
	If $S = \langle 3,\, 6k+\ell \rangle$ with $\ell \in \{1,2,4,5\}$ then
		$$O_n(S) = \begin{cases}
			\lfloor\frac{n}{3}\rfloor+1 & 0 \leq n \leq 3k + \frac{\ell}{2} - 1 \\
			\left\lfloor\frac{6k+\ell-1-n}{3}\right\rfloor+1 & 3k + \frac{\ell}{2} - 1 < n \leq 6k+\ell-1
		\end{cases}.$$
\end{lem}

Assuming Lemma \ref{overs3byG} we prove Theorem \ref{overs3}.

\begin{proof}[Proof of Theorem \ref{overs3}]
	Note that for $3k \leq n < 3k+\ell$ Lemma \ref{overs3byG} implies $O_n(S) = k+1$, so we can rewrite the expression for $O_n(S)$ as
		$$O_n(S) = \begin{cases}
			\lfloor\frac{n}{3}\rfloor+1 & 0 \leq n < 3k \\
			k+1 & 3k \leq n < 3k+\ell \\
			\left\lfloor\frac{6k+\ell-1-n}{3}\right\rfloor+1 & 3k+\ell \leq n \leq 6k+\ell-1
			\end{cases}. $$
	Now we find $O(S)$ by summing over $n$:
		$$O(S) = \sum_{n=0}^{6k+\ell-1} O_n(S) = 2\cdot3\cdot\frac{k(k+1)}{2} + \ell(k+1) = (3k+\ell)(k+1).$$
\end{proof}

We now prove the lemma through a careful consideration of Ap\'ery tuples.

\begin{proof}[Proof of Lemma \ref{overs3byG}]
	Fix $n$, and suppose $T \supseteq S$ with $g(T) = n$.
	%Since $3$ and $6k+\ell$ must be coprime, $\ell \in \{1,2,4,5\}$.
	We write $\ell = 3i+j$ where $i \in \{0,1\}$ and $j \in \{1,2\}$.  Let $m = 6k + \ell -1 - n$.  Since $m = g(T) - g(S)$ and $T \supseteq S,\ T$ is the union of $S$ together with $m$ gaps of $S$.  Let $6k+\ell-3p$ be the smallest element of $T$ that is congruent to $\ell$ modulo $3$. We see that $p \geq 0$ because $6k+\ell \in S \subseteq T$.  Since $T$ is closed under addition it include the elements $6k+\ell-3p+3t$ for all $t \geq 0$, so $T$ includes at least $p$ gaps of $S$.
	
	Since we know the smallest element of $T$ congruent to $\ell$ modulo $3$, the remaining $m-p$ elements of $T\sminus S$ are all congruent to $2\ell$ modulo $3$.
	The smallest element of $S$ congruent to $2\ell$ modulo $3$ is $12k+2\ell$, and hence the smallest element of $T$ congruent to $2\ell$ must be $12k+2\ell-3(m-p)$ to account for the correct number of gaps. Therefore, the Ap\'ery set of $T$ is $\{0,\, 6k+\ell-3p,\, 12k+2\ell-3(m-p)\}$.

	Such a numerical set $T$ is a numerical semigroup if and only if it satisfies the inequalities (\ref{ineq-semi1}) - (\ref{ineq-semi2}), which reduce to
		\begin{align*}
		2(6k+\ell-3p) &\geq 12k+2\ell-3(m-p)\\
		2(12k+2\ell-3(m-p)) &\geq 6k+\ell-3p
		\end{align*}
	which in turn give
		\begin{align}
		m &\geq 3p 		\label{ineq-m3p}\\
		6k + \ell + 3p &\geq 2m. 		\label{ineq-6k2m}
		\end{align}
	For fixed $n$ each value of $p$ gives a different numerical set $T$, and so $O_n(S)$ is equal to the number of values of $p$ satisfying both (\ref{ineq-m3p}) and (\ref{ineq-6k2m}).

	For $0 \leq n \leq 3k+\frac{\ell}{2}-1$ we have $3k+\frac{\ell}{2} \leq m \leq 6k+\ell-1$.
	Since $m = 6k+\ell-1-n$, the above inequalities can be rewritten
		\begin{align*}
		6k+\ell-1-n &\geq 3p \\
		6k+3p+\ell &\geq 12k+2\ell-2-2n
		\end{align*}
	which determine the interval
		\begin{equation} \label{ineq-p}
		\frac{6k+\ell-2-2n}{3}\ \leq\ p\ \leq\ \frac{6k+\ell-1-n}{3}.
		\end{equation}
	Since $n \leq 3k+\frac{\ell}{2}-1$, the lower bound for $p$ is greater than or equal to $0$.
	One can check that the distance between the bounds of $p$ given in (\ref{ineq-p}) is $\frac{n+1}{3}$, and by considering each case of $n$ modulo 3 one can see that there are always $\lfloor\frac{n}{3}\rfloor+1$ integers in this interval.
	Therefore, in this case $O_n(S) = \lfloor\frac{n}{3}\rfloor+1$.
	
	For $3k+\frac{\ell}{2}-1 < n \leq 6k+\ell-1$ we have $0 \leq m < 3k + \frac{\ell}{2}$.
	Because $2m < 6k+\ell \leq 6k+\ell+3p$ for any $p$, (\ref{ineq-6k2m}) holds for any $p \geq 0$.  So we need only count integer solutions to (\ref{ineq-m3p}).
	There are exactly $\lfloor{m/3}\rfloor + 1$ integers $p$ that satisfy (\ref{ineq-m3p}), so in this case $O_n(S) = \lfloor\frac{m}{3}\rfloor+1 = \left\lfloor\frac{6k+\ell-1-n}{3}\right\rfloor+1$.

\end{proof}

Theorem \ref{AndersonsTheorem} shows that for large $k$ there are about $6k^2$ simultaneous $(3,\,6k+\ell)$-cores.
From Theorem \ref{overs3} we know that about $3k^2$ of them are associated with semigroups.
Therefore, as $b$ approaches infinity, half of all $(3,b)$-cores correspond with numerical semigroups.  We give another interpretation of this result in Theorem \ref{3coreConj} in the next section.  

The case of $a = 4$, stated as Theorem \ref{overs4} in the introduction, is more complex but can be approached similarly.  We give a proof of only the case $\ell = 1$ here since the other cases are very similar.

%\begin{theorem} \label{overs4}
%Suppose that $S = \langle 4, 12k + \ell\rangle$ with $\ell \in \{1,3,5,7,9,11\}$.  Then $O(S)$ is given by the following chart:

%	\begin{center}
%	\begin{tabular}{c | c}
%	$\ell$	&	$O(S)$ \\
%	\hline
%	1	&	$24k^3 + 30k^2 + 11k + 1$ \\
%	3	&	$24k^3 + 42k^2 + 23k + 4$ \\
%	5	&	$24k^3 + 54k^2 + 39k + 9$ \\
%	7	&	$24k^3 + 66k^2 + 59k + 17$ \\
%	9	&	$24k^3 + 78k^2 + 83k + 29$ \\
%	11	&	$24k^3 + 90k^2 + 111k + 45$
%	\end{tabular}.
%	\end{center}

	%If $S = \langle 4,\, 12k+1 \rangle$, then $O(S) = 24k^3 + 30k^2 + 11k + 1$.
%\end{theorem}
%These computations agree with those in \cite{Hellus} where they are proven by different methods.  

A first approach to prove this might be to count oversemigroup by genus as we did for $a = 3$.
However, that approach does not work so nicely here; for example, the function that counts oversemigroups of $S = \langle 4,12k+1\rangle$ by genus is not unimodal. Instead, we count oversemigroups with Ap\'ery tuple $(x,n,y)$, where $n$ is fixed.  Let 
\[
O'_n(S) = \#\{T \supseteq S : T\text{ is a semigroup, } \Ap(T) = (x,n,y)\}.
\]

\begin{lem} \label{overs4byAp}
	If $S = \langle 4,\, 12k+1 \rangle$ then
		$$O'_n(S) = \begin{cases}
			(n+1)(6k - \frac{3n}{2} + 1) & 0 \leq n \leq 2k \\
			(n+1)(3k - \lceil \frac{n}{2} \rceil + 1) %
				+ \frac{1}{2}(3k - \lfloor \frac{n}{2} \rfloor)(3k - \lfloor \frac{n}{2} \rfloor + 1) & 2k < n \leq 6k
		\end{cases}.$$
\end{lem}

Using this lemma, we prove the $\ell = 1$ case of Theorem \ref{overs4}.

\begin{proof}[Proof of Theorem \ref{overs4} for $\ell =1$]
	Suppose $T$ is an oversemigroup of $S$ with $\Ap(T) = (x,n,y)$.  Since $6k\cdot 4 + 2 \in S$ we know $n \leq 6k$, which means $O(S) = \sum_{n=0}^{6k} O'_n(S)$.
	By Lemma \ref{overs4byAp} we have
	%\begin{IEEEeqnarray}{rCl}
	\[	\sum_{n=0}^{2k} O'_n(S)
			 =  \sum_{n=0}^{2k} (n+1)\left(6k - \frac{3}{2}n + 1\right) 
			 %		& = & (6k + 1)\frac{(2k+1)(2k+2)}{2} \nonumber\\
%			&&- \frac{3}{2}\cdot\left[\frac{1}{6}(2k)(2k+1)(4k+1) + \frac{1}%{2}(2k)(2k+1)\right] \nonumber \\
%		& = & (k + 1)(12k^2 + 8k + 1) - \frac{3}{2}(4k^2 + 2k)\left[\frac{2k}{3} + \frac{2}{3}\right] \\
%		& = & (k+1)(12k^2 + 8k + 1) - (4k^2 + 2k)(k+1) \\
%		& = & (k+1)(8k^2 + 6k + 1) \\
		 =  8k^3 + 14k^2 + 7k + 1 \label{overs4,n<2k},
		 \]
	%\end{IEEEeqnarray}
by a standard induction argument.

We also have
	\begin{IEEEeqnarray}{rCl}
		\sum_{n=2k+1}^{6k} O'_n(S)
			& = & \sum_{n=2k+1}^{6k} (n+1)(3k + 1)\nonumber 
				- (n+1)\left\lceil\frac{n}{2}\right\rceil + \frac{1}{2}{\Big(}3k - \left\lfloor\frac{n}{2}\right\rfloor{\Big)}\left(3k - {\Big \lfloor}\frac{n}{2}{\Big \rfloor} + 1\right) \nonumber \\
		& = & (3k + 1)(4k) + (\frac{6k(6k+1)}{2} - \frac{2k(2k+1)}{2})(3k+1) \nonumber\\
			&& - (9k^2 + 3k - k^2 - k) \nonumber \\
			&& - \frac{1}{6}\left[3k(3k+1)(24k+1) - k(k+1)(8k+1)\right] \nonumber \\
			&& + \frac{1}{2}\Big[(9k^2 + 3k)(4k) - (9k^2 - k^2) - 6k(9k^2 - k^2) \nonumber \\
				&&\qquad + \frac{1}{3}(3k(18k^2 + 1) - k(2k^2 + 1))\Big] \nonumber \\
		& = & 16k^3 + 16k^2 + 4k \label{overs4,n>2k},
	\end{IEEEeqnarray}
by a slightly more complicated induction argument.

	Adding (\ref{overs4,n<2k}) and (\ref{overs4,n>2k}) gives $O(S) = 24k^3 + 30k^2 + 11k + 1$.
	\end{proof}

We finish the argument by proving Lemma \ref{overs4byAp}.

\begin{proof}[Proof of Lemma \ref{overs4byAp}]
	Suppose $T \supseteq S$ with $\Ap(T) = (x,n,y)$.
	This Ap\'ery tuple must satisfy the inequalities (\ref{ineq-semi1}) - (\ref{ineq-over1}), which means that the following inequalities must hold:
	\begin{align}
		2x &\geq n 	\label{ineq-2xn} \\
		2y + 1 &\geq n \label{ineq-2y1n} \\
		x + n &\geq y 	\label{ineq-xny} \\
		y + n + 1 &\geq x \label{ineq-yn1x} \\
		x &\leq 3k \\
		n &\leq 6k	\label{ineq-n<6k}\\
		y &\leq 9k.
	\end{align}

	First, consider the case where $0 \leq n \leq 2k$.
	If $x \leq y$ then $x = y - c$ for some $c \in \{0,1,\ldots,n\}$.
	For any $x$ that satisfies $\lceil \frac{n}{2} \rceil \leq x \leq 3k$, the inequalities (\ref{ineq-2xn}) -- (\ref{ineq-yn1x}) are satisfied, so for each value of $c$ there are $3k - \lceil \frac{n}{2} \rceil + 1$ oversemigroups in this case.

	If $x > y$ then $x = y + c$ for some $c \in \{1, \ldots, n+1\}$.
	The above inequalities are satisfied if and only if $\lfloor \frac{n}{2} \rfloor + c \leq x \leq 3k$.
	Therefore, for each $c$ there are $3k - c - \lfloor \frac{n}{2} \rfloor + 1$ oversemigroups in this case.

	Summing over all values of $c$ for both $x \leq y$ and $x > y$, we see that
	\begin{align*}
		O'_n(S)
		&= \textstyle (n+1)(3k - \lceil\frac{n}{2}\rceil + 1) + (n+1)(3k - \lfloor\frac{n}{2}\rfloor + 1) - \frac{(n+1)(n+2)}{2}  \\
		&= \textstyle (n+1)(6k - \frac{3n}{2} + 1).
	\end{align*}

	Now consider the case where $2k < n \leq 6k$.
	If $x \leq y$ then $x = y - c$ for some $c \in \{0,1,\ldots,n\}$.
	As in the previous case, $(x,n,y)$ is a valid Ap\'ery tuple if and only if $\lceil\frac{n}{2}\rceil \leq x \leq 3k$, so for each $c$ there are $3k - \lceil\frac{n}{2}\rceil + 1$ oversemigroups in this case.

	If $x > y$ then $x = y + c$ for some $c \in \{1,2,\ldots,3k - \lfloor\frac{n}{2}\rfloor\}$.
	Again, (\ref{ineq-2xn}) -- (\ref{ineq-yn1x}) are satisfied if and only if $\lfloor\frac{n}{2}\rfloor + c \leq x \leq 3k$.
	Therefore, for each value of $c$ there are $3k - c - \lfloor\frac{n}{2}\rfloor + 1$ oversemigroups in this case.
	
	Summing over all values of $c$ for both $x \leq y$ and $x > y$, we obtain
	\begin{align*}
		O'_n(S) 
		&= \textstyle (n+1)(3k - \lceil\frac{n}{2}\rceil + 1) + (3k - \lfloor\frac{n}{2}\rfloor)(3k - \lfloor\frac{n}{2}\rfloor + 1) \\
			&\quad \textstyle - (3k - \lfloor\frac{n}{2}\rfloor + 1)(3k - \lfloor\frac{n}{2}\rfloor + 1)/2 \\
		&= \textstyle (n+1)(3k - \lceil\frac{n}{2}\rceil + 1) + (3k - \lfloor\frac{n}{2}\rfloor)(3k - \lfloor\frac{n}{2}\rfloor + 1)/2.
	\end{align*}

\end{proof}

%Analogously---by finding $O'_n(S)$ and summing over $n$---we can compute $O(\langle 4,\, 12k+\ell \rangle)$ for all $\ell$.
%If $S = \langle 4,\, 12k+\ell \rangle$, then

%	\begin{center}
%	\begin{tabular}{c | c}
%	$\ell$	&	$O(S)$ \\
%	\hline
%	1	&	$24k^3 + 30k^2 + 11k + 1$ \\
%	3	&	$24k^3 + 42k^2 + 23k + 4$ \\
%	5	&	$24k^3 + 54k^2 + 39k + 9$ \\
%	7	&	$24k^3 + 66k^2 + 59k + 17$ \\
%	9	&	$24k^3 + 78k^2 + 83k + 29$ \\
%	11	&	$24k^3 + 90k^2 + 111k + 45$
%	\end{tabular}
%	\end{center}

Comparing Theorem \ref{overs4} with Theorem \ref{AndersonsTheorem}, we see that for large $k$ there are approximately $72k^3$ simultaneous $(4,\,12k+\ell)$-cores, of which about $24k^3$ are associated with semigroups.
Thus, as $b$ approaches infinity, one third of all $(4,b)$-cores correspond with numerical semigroups.

We compare the behavior of $C(a,b)$---the total number of $(a,b)$-cores---with $O(\langle a, b \rangle)$ for large values of $b$:
	\begin{center}
	\begin{tabular}{c | c}
		$a\ $ & $\ \lim_{b\to\infty}{O(\langle a, b \rangle)}/{C(a,b)}$ \\
		\hline
		$2\ $ & 1 \\
		$3\ $ & $1/2$ \\
		$4\ $ & $1/3$
	\end{tabular}.
	\end{center}
We can ask for the behavior of this ratio for larger values of $a$.  As a degree $a-1$ polynomial in $b$, the leading coefficient of $C(a,b)$ is $\frac{1}{a!}$.
 Theorem \ref{HellusTheorem} of Hellus and Waldi shows that this ratio is between $\frac{1}{(a-1)!}$ and $\frac{1}{a-1}$.
Therefore
	$$ \lim_{a\to\infty}\lim_{b\to\infty} \frac{O(\langle a, b \rangle)}{C(a,b)} \le \lim_{a\to\infty} \frac{1}{a-1} = 0 . $$
These results can be interpreted as special cases of Problem \ref{Prob1} since the leading coefficient of these quasipolynomials are closely related to the volumes of the $(a,b)$-core polytopes of the previous section.

%%%%%%%%%%%%%%%%%%%%%%%%%%%%%%%%%%%%%%%%%%%%%%%%%%%%%%%%%%%%%%%%

\section{Conjugate partitions and symmetric numerical sets}

Recall that a numerical set $T$ with Frobenius number $F$ is symmetric if and only if for each $i \in [0,F]$ exactly one of $i,F-i$ is in $T$ and that the conjugate of a partition $\lambda$ is the partition $\widetilde{\lambda}$ that we get from interchanging the rows and columns of the Young diagram of $\lambda$.  Our first goal is to relate these two concepts.  We then focus on the particular case of $3$-core partitions and their conjugates.

\begin{proposition} \label{symIffSelfConj}
	A numerical set $T$ is symmetric if and only if $\varphi(T)$ is a self-conjugate partition.
\end{proposition}

In order to prove this proposition we give a characterization of the numerical set associated to $\widetilde{\lambda}$ under the bijection $\varphi$.  The \emph{dual} of a numerical set $T$ with Frobenius number $F$ is the numerical set $T^* = \left\{u \in \Z\ :\ F-u \not\in T\right\}$.  A numerical set and its dual have the same atom monoid and it is clear that a numerical set is symmetric if and only if it is equal to its dual. For additional background on this concept, see Section 1 of \cite{Antokoletz}.  By considering pairs $i,F-i$ and whether or not they are elements of $T$ we get the following characterization of $T^*$:
\[
T^* = \{F-u\ : u \in \Z \sminus T\}.
\]

%\nathan{This follows directly from Proposition 1 of Marzuola and Miller.}

\begin{proposition} \label{conjRights}
	Suppose $T$ is a numerical set with Frobenius number $F$ and $\varphi(T) = \lambda$. The numerical set associated with $\widetilde{\lambda}$ is $T^*$.
%		\[
%		 \varphi^{-1}(\widetilde{\lambda}) = \{ F - u : u \in \Z \sminus T\} %\cup \{F+1,F+2,\ldots\right\}.
%		\]
\end{proposition}

\begin{proof}
It is easy to see from the definition of hook length that $H(\lambda) = H(\widetilde{\lambda})$ and $F\left(\varphi^{-1}(\widetilde{\lambda})\right) = F\left(\varphi^{-1}({\lambda})\right)$.  We now label the profile of $\lambda$ in reverse order, starting with $F$ and counting down.
	The up-steps of this labeling are of the form $F - u$ for $u \notin T$ and are exactly the right steps of $\widetilde{\lambda}$.	
	\end{proof}
	
	We now use this characterization to prove Proposition \ref{symIffSelfConj}.  This is both a slight generalization of Proposition 4.4 in \cite{Rosales} and a slight reframing of Proposition 1 of \cite{Marzuola}, since it is now clear that a partition is self-conjugate if and only if the corresponding numerical set is equal to its dual.
\begin{proof}[Proof of Proposition \ref{symIffSelfConj}]
	Let $F$ be the Frobenius number of $T$.  We need only show that $T$ is symmetric if and only if $T = \{F - u: u \in \Z \sminus T\}$.  
	
	First suppose $T$ is symmetric and $u \in \Z \sminus T$. Then $F - u \in T$ by definition.  If $F - u \in T$, then $F - (F - u) = u \notin T$.
	
	Conversely, suppose $T = \{F - u: u \notin T\}$.  If $x \in T$, then $x = F - u$ for some $u \not\in T$.  Now $ u = F-x$ and we see that $T$ is symmetric.
\end{proof}

%Given a partition $\lambda$, the \emph{conjugate} of $\lambda$, written $\widetilde{\lambda}$, is the partition one gets by interchanging the rows and columns of the Young diagram of $\lambda$.
%It is easy to see that $H(\lambda) = H(\widetilde{\lambda})$ because of the symmetric definition of a hook length, though the numerical sets associated to $\lambda$ and $\widetilde{\lambda}$ are not necessarily the same.
%In the case that $\lambda = \widetilde{\lambda}$, we call $\lambda$ a \emph{self-conjugate} partition.

%Let $T$ be a numerical set with odd Frobenius number $F$.
%We call $T$ a \emph{symmetric} numerical set if for each $a$, exactly one of %$\{a,\ F-a\}$ is in $T$.
%As we will see, self-conjugate partitions correspond exactly to symmetric numerical sets.

%To prove this, we first find the more general relation of the numerical set associated to a partition $\lambda$ with the numerical set associated to $\widetilde{\lambda}$.

We give an example to illustrate this process.
		
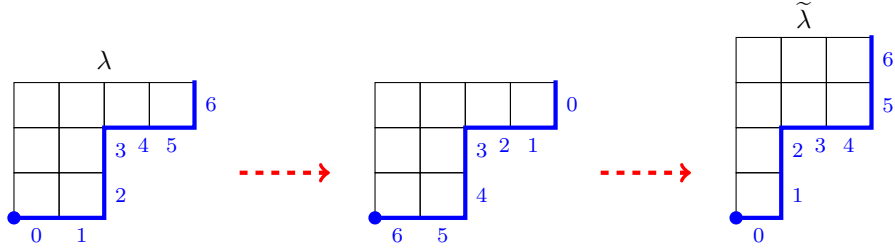
\begin{center}
			\begin{tikzpicture}[scale = 0.6]
			% T = \{0,1,4,5,7,\to\}
			% \lambda = (4, 2, 2)
			
				% 1 - rim labelled for original partition
				\begin{scope}[xshift = 0cm]
					\foreach \x in {0,1,2,3}
						{ \draw (\x,2) rectangle ++(1,1);}
					\foreach \x in {0,1}
						{ \draw (\x,1) rectangle ++(1,1);}
					\foreach \x in {0,1}
						{ \draw (\x,0) rectangle ++(1,1);}
					
					\begin{scope}[blue, ultra thick]
						\filldraw (0,0) circle [radius = 0.1];
						\draw[font = \footnotesize]
							(0,0) --
							node [below] {0} ++(1,0) --
							node [below] {1} ++(1,0) -- 
							node [right] {2} ++(0,1) --
							node [right] {3} ++(0,1) --
							node [below right] {4} ++(1,0) --
							node [below] {5} ++(1,0) --
							node [right] {6} ++(0,1.05);
					\end{scope}
					
					\node at (2,3.5) {$\lambda$};
				\end{scope}
				
				% 2 - rim labelled starting with F(T)
				\begin{scope}[xshift = 8cm]
					\foreach \x in {0,1,2,3}
						{ \draw (\x,2) rectangle ++(1,1);}
					\foreach \x in {0,1}
						{ \draw (\x,1) rectangle ++(1,1);}
					\foreach \x in {0,1}
						{ \draw (\x,0) rectangle ++(1,1);}
					
					\begin{scope}[blue, ultra thick]
						\filldraw (0,0) circle [radius = 0.1];
						\draw[font = \footnotesize]
							(0,0) --
							node [below] {6} ++(1,0) --
							node [below] {5} ++(1,0) -- 
							node [right] {4} ++(0,1) --
							node [right] {3} ++(0,1) --
							node [below right] {2} ++(1,0) --
							node [below] {1} ++(1,0) --
							node [right] {0} ++(0,1.05);
					\end{scope}
				\end{scope}
				
				% 3 - conjugate with rim labels
				\begin{scope}[xshift = 16cm]
					\foreach \x in {0,1,2}
						{ \draw (\x,3) rectangle ++(1,1);}
					\foreach \x in {0,1,2}
						{ \draw (\x,2) rectangle ++(1,1);}
					\foreach \x in {0}
						{ \draw (\x,1) rectangle ++(1,1);}
					\foreach \x in {0}
						{ \draw (\x,0) rectangle ++(1,1);}
					
					\begin{scope}[blue, ultra thick]
						\filldraw (0,0) circle [radius = 0.1];
						\draw[font = \footnotesize]
							(0,0) --
							node [below] {0} ++(1,0) --
							node [right] {1} ++(0,1) -- 
							node [right] {2} ++(0,1) --
							node [below right] {3} ++(1,0) --
							node [below] {4} ++(1,0) --
							node [right] {5} ++(0,1) --
							node [right] {6} ++(0,1.05);
					\end{scope}
					
					\node at (1.5,4.5) {$\widetilde{\lambda}$};
				\end{scope}
				
				% arrows
				\begin{scope}[ultra thick, dashed,  red, ->]
					\draw (5,1) -- (7,1);
					\draw (13,1) -- (15,1);
				\end{scope}
			\end{tikzpicture}
		\end{center}

The conjugate of an $a$-core partition $\lambda$ is also an $a$-core and we have seen how Ap\'ery tuples map such partitions to $\N^{a-1}$ so it is natural to ask how conjugation acts on $\N^{a-1}$. In other words, we wish to find the Ap\'ery tuple of $\varphi^{-1}(\widetilde{\lambda})$ given the Ap\'ery tuple of $\varphi^{-1}(\lambda)$.

%In general this is difficult to determine, however $3$-core partitions behave well in this way.

\begin{proposition} \label{conjugateApery}
	Suppose that $\lambda$ is a partition with corresponding numerical set $T = \varphi^{-1}(\lambda)$ with Frobenius number $F$ such that the Ap\'ery tuple of $T$ is $\Ap(T) = (x_1,\ldots,x_{a-1})$ and $F \equiv \ell \pmod{a}$.
	%Let $S = \varphi^{-1}(\widetilde{\lambda})$ be the numerical set corresponding to the conjugate partition of $\lambda$, and $\Ap(S) = (x'_1,\ldots,x'_{a-1})$.
	Then the Ap\'ery tuple of $T^* = \varphi^{-1}(\widetilde{\lambda})$ is $\Ap(T^*) = (x'_1,\ldots,x'_{a-1})$ where
		$$ x'_i =
			\begin{cases}
				x_\ell - x_{\ell - i} & i < \ell \\
				x_\ell & i = \ell \\
				x_\ell - x_{a + \ell - i} - 1 & i > \ell
			\end{cases}. $$
\end{proposition}

\begin{proof}
	Recall from Proposition \ref{conjRights} that $S = \{F - u : u \in \Z\sminus T\}$.
	Thus
		\begin{align*}
			ax'_i + i
			&= \min\{F - u : u \notin S, F - u \equiv i \pmod{a}\} \\
			&= F - \max\{u \notin S : u \equiv \ell - i \pmod{a}\}.
		\end{align*}
	By the definition of the Ap\'ery tuple,
	\[
\max\left\{u \notin S : u \equiv \ell - i \pmod{a}\right\} = 
\begin{cases} 
 a(x_{\ell-i}-1) +(\ell-i) & i < \ell \\
-a & i = \ell \\
a x_{a + \ell-i} + (\ell-i) & i > \ell
\end{cases}.
	\]		
Noting $F = a(x_{\ell}-1)+ \ell$ completes the proof.

	%If $i < \ell$, then the last term is $a(x_{\ell - i}-1) + (\ell - i)$; if $i = \ell$, then it is $-a$; if $i > \ell$, then it is $ax_{\ell - i} + (\ell - i)$.
	
\end{proof}

Proposition \ref{conjugateApery} allows us to prove a theorem unique to $3$-core partitions that relates them to numerical semigroups.

\begin{theorem}	\label{3coreConj}
	Given a $3$-core partition $\lambda$, either $\varphi^{-1}(\lambda)$ or $\varphi^{-1}(\widetilde{\lambda})$ is a numerical semigroup.
\end{theorem}

\begin{proof}
	Let $T = \varphi^{-1}(\lambda)$ and $S = \varphi^{-1}(\widetilde{\lambda})$.
	Suppose that $\Ap(T) = (x_1,x_2)$.
	Recall that $T$ is a numerical semigroup if and only if it satisfies the inequalities (\ref{ineq-semi1}) - (\ref{ineq-semi2}), which here reduce to
		\begin{align}
		2x_1 \geq x_2 \label{ineq-3core-semi1}\\
		2x_2 + 1 \geq x_1 \label{ineq-3core-semi2}.
		\end{align}
	Notice that at least one of these must be true.
	
	If (\ref{ineq-3core-semi1}) fails, then $x_1 < x_2$, and so by Proposition \ref{conjugateApery}, $\Ap(S) = (x_2 - x_1, x_2)$.
	Using the fact that $2x_1 < x_2$ and $2x_2 + 1 \geq x_1$, we see that (\ref{ineq-semi1}) and (\ref{ineq-semi2}) are both satisfied for $S$, and hence $S$ is a numerical semigroup.
	
	If instead (\ref{ineq-3core-semi2}) fails, then $x_1 > x_2$, so by Proposition \ref{conjugateApery}, $\Ap(S) = (x_1, x_1 - x_2 - 1)$.
	As before, we use the fact that $2x_1 \geq x_2$ and $2x_2 + 1 < x_1$ to show that (\ref{ineq-semi1}) and (\ref{ineq-semi2}) are satisfied for $S$.
	So $S$ is again a numerical semigroup.
	
\end{proof}

From Theorem \ref{overs3} we know the number of numerical semigroups containing $\langle 3, 6k + \ell \rangle$, so using Theorem \ref{3coreConj} we can determine the number of these semigroups that are symmetric.
A symmetric numerical semigroup is sent to a self-conjugate partition under $\varphi$, so this number is also equal to the number of self-conjugate $(3, 6k + \ell)$-core partitions associated to numerical semigroups.

\begin{cor} \label{symOvers3}
	The number of symmetric numerical semigroups containing $\langle 3, 6k+\ell \rangle$ is
		$$3k + \frac{3\ell}{2} - \frac{\ell^2}{6} - \frac{1}{3}.$$
\end{cor}

\begin{proof}
	%For convenience we will write $b = 6k+\ell$.
	By Theorem \ref{3coreConj}, for any $(3,6k+\ell)$-core partition $\lambda$, either $\varphi^{-1}(\lambda)$ or $\varphi^{-1}(\widetilde{\lambda})$ is a semigroup.
	Therefore if we double count the number oversemigroups of $S$ we will have counted every non-self-conjugate $(3,6k+\ell)$-core exactly once, and we will have counted the number of self-conjugate $(3,6k+\ell)$-cores twice.
	Therefore, the number of self-conjugate $(3,6k+\ell)$-core partitions, which is the same as the number of symmetric oversemigroups of $S$ by Theorem \ref{3coreConj}, is
		$$2 \cdot O(\langle 3, 6k+\ell \rangle) - C(3,6k+\ell) = 3k + \frac{3\ell}{2} - \frac{\ell^2}{6} - \frac{1}{3}.$$
	
\end{proof}

% % % % % % % % % % % % % % % % % % % % % % % % % % % % % % % %

\section{Counting partitions with a given hook set} \label{sec-P(S)}

In much of this paper we have studied statistical questions about distribution of sizes of the finite set of simultaneous $(a,b_1,\ldots, b_m)$-core partitions.  In this section we turn towards another finite collection of partitions, those which have the same hook set.  By Proposition \ref{HooksAreComplementOfAtom} the hook set of any partition is the complement of some numerical semigroup $S$.  Our goal is to understand the set of partitions sharing a given hook set and what properties of the underlying semigroup influence the size of this set.
%In the study of hook sets of partitions it is natural to ask how many partitions share a given hook set.
%Recall from Proposition \ref{HooksAreComplementOfAtom} that the hook set of any partition is the complement of some numerical semigroup.
Therefore, we rephrase this question as: Given a numerical semigroup $S$, for how many partitions $\lambda$ is $H(\lambda) = \N \sminus S$?
We call this number $P(S)$.  By our discussion of the bijection $\varphi$ in Section \ref{SetsPartitions}, this is equivalent to counting the number of numerical sets with atom monoid $S$.

This problem has been considered by Marzuola and Miller in \cite{Marzuola} where they call it the \emph{Anti-Atom Problem}.  They give constraints on numerical sets sharing the same atom monoid $S$ in terms of the dual numerical set $S^*$.

\begin{proposition}[Proposition 1 in \cite{Marzuola}]\label{MMProp1}
Suppose that $S$ is a numerical semigroup and that $T$ is a numerical set with $A(T) = S$.  Then $S \subseteq T \subseteq S^*$.
\end{proposition}
We note that the description of $T^*$ given directly above Proposition \ref{conjRights} also gives a way to prove this fact in terms of partitions with a given hook set.  

In cases where the gap between $S$ and $S^*$ is well-understood this result gives a strong characterization of the numerical sets with atom monoid $S$.  A numerical semigroup $S$ is \emph{pseudosymmetric} if $F(S)$ is even and for every $i \in [0,F(S)/2)$ exactly one of $i, F(S)-i$ is in $S$. 
\begin{cor}[Corollary 2 in \cite{Marzuola}]\label{MMcor2}
A numerical monoid $S$ with Frobenius number $F$ is symmetric if and only if there is just one numerical set (which must be $S$ itself) whose atom monoid is $S$.  Equivalently, $P(S) = 1$ if and only if $S$ is symmetric.

If $S$ is a pseudosymmetric numerical semigroup then there are precisely two numerical sets (which must be $S$ and $S^*$) whose atom monoid is $S$.  Equivalently, if $S$ is pseudosymmetric then $P(S) = 2$.
\end{cor}
We note that the first part of the corollary is equivalent to Proposition \ref{symIffSelfConj}.  As we will see below, the converse of the second statement does not hold and it seems difficult to give a complete classification of numerical semigroups $S$ with $P(S) = 2$.

We give a bound for $P(S)$ in terms of how far away $S$ is from being symmetric.  A \emph{missing pair} of $S$ is a pair of elements $i,F(S)-i$ with $i \le F(S) - i$ such that neither element is in $S$. Note that when $F(S)$ is even we have the degenerate missing pair consisting of the single element $F(S)/2$.  Let $M(S)$ denote the union of the set of missing pairs of $S$.

%To study the behavior of $P(S)$ further, we study its behavior with respect to the behavior of the missing pairs of $S$.
%A $\emph{missing pair}$ of $S$ is an ordered pair $(a, F - a)$ with $a < F - a$ and both $a \notin S$ and $F - a \notin S$.
%Note that this condition forces $0 < a < F - a$.
%Let $M(S)$ be the number of missing pairs of the semigroup $S$.

%This approach of considering the missing pairs of $S$ is motivated by the following lemma:

\begin{lem}\label{SunionMPs}
For a numerical semigroup $S$ we have $S^* = S \cup M(S)$.
\end{lem}

%\begin{lem} \label{SunionMPs}
%	Let $S$ be a numerical semigroup, and let $\cc{M} = \{a, F - a : a < F - a,\ a \notin S,\ F - a \notin S\}$ be the set of elements in missing pairs of $S$.
%	Let $U = \{F - u: u \notin S\}$, the numerical set associated to the conjugate partition of $\varphi(S)$.
%	Then $S \cup \cc{M} = U$.
%\end{lem}

\begin{proof}
	%As a corollary of Proposition \ref{P(S)bounds} we have $S \subseteq U$.
Let $F$ be the Frobenius number of $S$, which is also the Frobenius number of $S^*$.  We need only consider elements less than $F$.  We first recall that 
\[
S^* = \{F-u\ : u \in \Z \sminus S\}.
\]
If $n,F-n$ is a missing pair of $S$ then $F - n \in S^*$ since $n \notin S$, and $n = F - (F - n) \in S^*$ as $F - n \notin S$.  Therefore $M(S) \subset S^*$.
	
	For the reverse inclusion, suppose that $n = F - u \in S^*$, where $u \in \N \sminus S$.  If $n \notin S$ then $u,n \in M(S)$.  If $n \in S$ then $u = F-n \not\in S^*$.  We conclude that $S^* = S \cup M(S)$.
\end{proof}

We could replace every instance of $M(S)$ with $S^*\sminus S$ but we choose to keep the notation of missing pairs since it is more descriptive.

\begin{cor}\label{PSBound}
For a numerical semigroup $S,\ P(S) \le 2^{|M(S)|}$.
\end{cor}

\begin{proof}
A numerical semigroup $T$ with hook set $\N \sminus S$ is the union of $S$ with some subset of $M(S)$.
\end{proof}
Since $M(S)$ is empty for a symmetric semigroup and consists of a single element for a pseudosymmetric semigroup, this gives another proof of Corollary \ref{MMcor2}.

Now that we understand semigroups for which $|M(S)| \le 1$, we consider those for which $|M(S)| = 2$. 
\begin{proposition}
	For a numerical semigroup $S$ with $|M(S)| = 2,\ P(S) \in \{2,3\}$.
\end{proposition}

\begin{proof}
	Let $F$ be the Frobenius number of $S$ and $a,F-a$ be the missing pair of $S$ where $a< F/2$.  Since $S$ is not symmetric, $P(S) \ge 2$.  By Corollary \ref{PSBound} we need only show that $P(S) \neq 4$. By the lemma above we need only show that $A\left(S \cup \{F-a\}\right) \neq S$.
	
	We argue by contradiction.  Suppose $A\left(S \cup \{F-a\}\right) = S$.  Since $F-a \not\in A(S \cup \{F-a\})$ there is some $n \in S \cup \{F-a\}$ such that $n + F-a \not\in S$.  Since $F-a > F/2$ we cannot have $n = F-a$.  So $n \in S$ and $n \not\in A\left(S \cup \{F-a\}\right)$, which is a contradiction.
\end{proof}
We note that both cases $P(S) =2$ and $P(S) =  3$ are possible.  For example, $S = \{0,4,\to\}$ has $A(S \cup \{1\}) = A(S \cup\{1,2\}) = S$ and $P(S) = 3$, and $S = \{0,3,6,\to\}$ has $M(S) = \{1,4\}$ and $P(S) = 2$.

Corollary \ref{PSBound} shows that if $|M(S)|$ is small then $P(S)$ is small.  We give a family of semigroups showing that the converse does not necessarily hold.
\begin{proposition} \label{missingPairFailure}
	For odd $N \in \N$ with $N \ge 11$, let $R_N$ be the numerical semigroup
		$$R_N = \{0, \textstyle\frac{N+1}{2}\} \cup E_N \cup \{ N+1,\ N+2, \ldots\}.$$
	where $E_N$ is the set of even numbers in $\left(\frac{N+1}{2}, N-1\right)$.
	%$n$ such that $\frac{N+1}{2} \leq n < N-1$. %(so that $N-1 \notin R_N$).
	We have that $P(R_N) = 2$ but $|M(R_N)| = 2 \left\lceil \frac{N - 1}{4} \right\rceil$.
\end{proposition}

\begin{proof}
The statement about $M(R_N)$ follows easily from the fact that $F(R_N) = N$ and the observation that every missing pair except $\{1,N-1\}$ is uniquely determined by an odd number in $\left(\frac{N+1}{2}, N-1\right)$.

Since $R_N$ is not symmetric we see that $R_N \subsetneq R_N^*$ and $A(R_N^*) = R_N$.  Suppose $T \neq R_N$ is a numerical set with $A(T) = R_N$.  By Proposition \ref{MMProp1} we have $T \subseteq R_N^*$.  We show that $P(R_N) = 2$ by showing that 
\[
R_N^* = \{N-u\ :\ u \in \Z \sminus R_N\} \subseteq T.
\] 

%	Suppose $T \neq R_N$ is a numerical set such that $A(T) = R_N$.
%	By Proposition \ref{P(S)bounds}, to show $P(R_N) = 2$ we only need to show that $\{N - u: u \notin R_N\} \subseteq T$.
	
	Notice $\bb{N} \sminus R_N = \{1,\ldots,\frac{N-1}{2}\}\ \cup O_N \cup \{N-1\}$, where $O_N$ is the set of odd numbers in $[\frac{N+3}{2},N)$.
	Thus, $R_N^* = \{N - u : u \notin R_N\}$ is the set of even numbers in $[0,\frac{N-3}{2}]$ together with $\{1, \frac{N+1}{2},\ldots,N-1\}$ and $\{N+1,N+2,\ldots\}$.
	
	Since $T \neq R_N$ there exists some $t \in \left(R_N^*\sminus R_N\right) \cap T$.  Either $1\in T,\ T$ contains an even number in $(0,\frac{N-3}{2}],\ T$ contains $N-1$, or $T$ contains an odd number in $(\frac{N+1}{2},N-1)$.  In each case we will show that $T = R_N^*$.
	
	If $1 \in T$ then $A(T) = R_N$ implies that $\{\frac{N+1}{2},\ldots,N-2\} \subset T$ since $R_N$ contains the even numbers in this range.
	However, the odd numbers in this range are not in $A(T)$, meaning that for each $N - 2k$ there is some $s_k \in T$ such that $N-2k + s_k \notin T$.
	The only possibility is that $N-2k +s_k = N$, so $s_k = 2k$, meaning $T$ must contain the even numbers in $[0,\frac{N-3}{2}]$ and $T = R_N^*$.
	
	Suppose $T$ contains an even number $t \in (0,\frac{N-3}{2}]$.  Since $A(T)$ contains all even numbers in $(\frac{N+1}{2}, N-1)$ we see that $N-1-t \in A(T)$.  Since $t + (N-1-t) = N-1$, we have $N-1 \in T$.  However, $N-1\not\in A(T)$, so there must exist $u \in T$ with $N-1+u \not\in T$.  The only possibility is $u = 1$, which by the argument of the previous paragraph shows $T = R_N^*$.  
	
	Now suppose that $N-1\in T$.  Just as in the previous paragraph, since $N-1\not\in A(T)$ we see that $1\in T$ and $T = R_N^*$.
	
	Finally, suppose $T$ contains an odd number $t \in (\frac{N+1}{2},N-1)$.  Since $t \not\in A(T)$ there exists $u\in T$ such that $t+u \not\in T$.  Since $R_N$ contains all even numbers in $(\frac{N+1}{2},N-1)$ we either have $t+u$ equal to an odd number in $(\frac{N+1}{2},N-1)$ or equal to $N-1$.  In the first case $u$ is an even number in $(0,\frac{N-3}{2}]$, putting us in the situation described above, and we conclude $T = R_N^*$.  If $t+u = N-1$ then $u$ is an odd number in $(0,\frac{N-3}{2})$.  Since $u \in R_N^*$ we must have $u = 1$, putting us in the situation above, and we conclude that $T = R_N^*$.

%	Suppose $t \leq (N-3)/2$ and $t \neq 1$.
%	Then $t$ must be even since $T \subseteq \{N - u : u \notin R_N\}$.
%	Also $N - 1 - t \geq (N+1)/2$ and is even, so $N - 1 - t \in R_N$.
%	Since $A(T) = R_N$ and $(N - 1 - t) + t = N - 1$ this implies $N - 1 \in T$.
%	But $N - 1 \notin A(T) = R_N$, so there must be some $u \in T$ such that $N - 1 + u \notin R_N$.
%	The only possibility is $N - 1 + u = N$, so $u = 1 \in T$.
	
%	Suppose instead that $t \geq (N+1)/2$.
%	If $t = N - 1$ then by the same logic as the previous paragraph we have $1 \in T$, so assume $t \neq N - 1$.
%	Since $t \notin R_N$ this means $t$ is odd, and because $t \notin A(T)$ we have $t + u \notin T$ for some $u \in T$.
	
%	If $t + u = N - 1$ then $u$ is odd and $u \leq (N-3)/2$, a contradiction since such a number is not in $\{N - u: u \notin R_N\}$ but $u \in T \subseteq \{N - u: u \notin R_N\}$.
%	Thus, $t + u \notin T$ and is odd, meaning that $u$ is even with $u \leq N - (N+1)/2 = (N-1)/2$
	
%	If $u = (N-1)/2$ then again we have a contradiction because $(N-1)/2$ is not in $\{N - u: u \notin R_N\}$.
%	Therefore we have $u \leq (N-3)/2$ and $u \in T$ is even, so we reduce to the previous case, so $1 \in T$.
	
%	To show $M(R_N)$ is unbounded, simply notice that
%		$$ M(R_N) = \left\lceil \frac{N - 1}{4} \right\rceil . $$
\end{proof}

We now use a main result of Marzuola and Miller \cite{Marzuola} to study the opposite extreme, semigroups $S$ for which $M(S)$ is as large as possible given the genus of $S$.
\begin{proposition}\label{BigAtomMonoid}
	Let $S_{N} = \{0, N+1, N+2, \cdots\}$ be the numerical semigroup where $H(\varphi(S)) =  \{1, 2, \cdots, N\}$.  Then $P(S_{N}) \sim {c \cdot 2^{N}}$, where $c$ is a constant approximately equal to $0.2422$.
\end{proposition}

\begin{proof}
	Let $\gamma_{N}$ be the ratio of the number of numerical sets with atom monoid $S_N$ to the number of numerical sets with Frobenius number $N$.
	One of the main results of \cite{Marzuola} is that the sequence $\{\gamma_{N}\}$ is decreasing and converges to a number $\gamma \approx 0.4844$ with accuracy to within $0.0050$.  Numerical sets with Frobenius number $N$ are in bijection with subsets of $\{1,\ldots, N-1\}$, so there are $2^{N-1}$ of them.  Therefore, $P(S_N) = \gamma_N \cdot 2^{N-1}$, completing the proof.
	
	%Recall that $P(S_N)$ is the number of numerical sets with atom monoid $\{0,N+1,N+2,\cdots\}$.  We know that the number of numerical sets with Frobenius number $N$ is $2^{N-1}$ because for each subset of $\{1,\ldots,N-1\}$ there is such a numerical set.  Therefore, $P(S_N) = \gamma_N \cdot 2^{N-1}$, so if $c = \gamma/2$ then $P(S_N) \sim \gamma \cdot 2^{N-1} = c\cdot 2^N$.
	
\end{proof}

%	This case is an extreme example, and in fact $P(S)$ does not grow exactly with the number of missing pairs.
%	The next proposition establishes a family of numerical semigroups $S$ such that $P(S) = 2$, but the number of missing pairs is unbounded.
	
We end this section by giving a link between the study of partitions with a given hook set and partitions that come from numerical semigroups under $\varphi$.

%\section{The asymptotic behavior of the number of partitions corresponding to semigroups}

%Let $p(n)$ be the total number of partitions of $n$ and let $M(n)$ be the number of partitions of $n$ that come from semigroups via the map $\varphi$.
%We are interested in the asymptotic behavior of $\frac{M(n)}{p(n)}$.
%By translating results by Marzuola and Miller, combined with a result by J. Backelin, we have the following related result:

\begin{proposition}\label{PartitionsByMaxHook}
	Let $S(N)$ be the number of partitions with maximum hook length $N$ corresponding via $\varphi$ to numerical semigroups and let $T(N)$ be the number of partitions with maximum hook length $N$.
	Then,
		$$\lim_{n \to \infty} \frac {S(N)}{T(N)} = 0.$$
\end{proposition}

We use a bound due to Backelin on the number of numerical semigroups with given Frobenius number.
\begin{theorem}[Theorem 1.1 in \cite{Backelin}]
The number of numerical semigroups $S$ with Frobenius number $N$ is at most $4 \cdot 2^{\left\lfloor{{(N-1)}/{2}}\right\rfloor}$.
\end{theorem}

\begin{proof}[Proof of Proposition \ref{PartitionsByMaxHook}]
	Partitions with maximum hook length $N$ are in bijection with numerical sets with Frobenius number $N$, so $T(N) = 2^{N-1}$.  Similarly, $S(N)$ is the number of numerical semigroups with Frobenius number $N$.  By Backelin's theorem, 
	$$ \frac{S(N)}{T(N)} \leq 4 \cdot 2^{ \left \lfloor{{(N-1)}/{2}}\right \rfloor - (N-1)} \leq 4\cdot 2^ {-(N-1)/2}, $$
and therefore
		$$\lim_{n \to \infty} \frac {S(N)}{T(N)} = 0.$$

%	In \cite{Backelin}, Backelin showed that this number satisfies $S(N) \leq 4 \cdot 2^{\left\lfloor{{(N-1)}/{2}}\right\rfloor}$.
%	Therefore,
%		$$ \frac{S(N)}{T(N)} \leq 4 \cdot 2^{ \left \lfloor{{(N-1)}/{2}}\right \rfloor - (N-1)} \leq 2^ {-N/2}, $$
%	and hence
%		$$\lim_{n \to \infty} \frac {S(N)}{T(N)} = 0.$$

\end{proof}

% % % % % % % % % % % % % % % % % % % % % % % % % % % % % % % %

\section{Further questions}

We begin by returning to Problem \ref{Prob1}.  The simultaneous $(a,b_1,\ldots, b_m)$-core partitions are in bijection with integer points in a certain polytope.  We would like to be able to give formulas for the number of lattice points in this polytope and also for its volume.  Understanding these questions gives one approach to determining the correct leading coefficient of the quasipolynomial given in the second part of Theorem \ref{HellusTheorem} of Hellus and Waldi \cite{Hellus}.  The size of a partition corresponding to a lattice point comes from evaluating the quadratic function $F_a(x_1,\ldots, x_{a-1})$ of Section \ref{Correspondence}.  Under what circumstances can we give a nice description of the lattice point of this polytope on which this function takes its maximum value?  When can we give a nice expression for the average value of this function taken over all of these lattice points or give even more detailed statistical information about this set of values? We would like to have a better understanding of how tools from Ehrhart theory can be used to study these problems.

It seems likely that most partitions are not associated to numerical semigroups by the bijection $\varphi$, as most numerical sets are not closed under addition.  A subtle difficulty in addressing these types of questions comes from the fact that making statements about `most' partitions or `most' numerical sets requires an ordering.  The most natural ordering on partitions, in our opinion, is by size.  Proposition \ref{PartitionsByMaxHook} shows that if we instead order partitions by the size of their maximum hook length our intuition is correct.
\begin{conjecture}
Let $P(n)$ be the number of partitions of size at most $n$ and let $S'(n)$ be the number of these that are associated to numerical semigroups under $\varphi$.  Then
\[
\lim_{n \rightarrow \infty} \frac{S'(n)}{P(n)} = 0.
\]
\end{conjecture}

We ask a similar question for $a$-cores.
\begin{problem}
Let $a \ge 2$ be a positive integer, $P_a(n)$ be the number of $a$-core partitions of size at most $n$, and $S'_a(n)$ be the number of these partitions associated to numerical semigroups under $\varphi$.  Determine
\[
\lim_{n \rightarrow \infty} \frac{S_a'(n)}{P_a(n)}
\]
as a function of $a$.
\end{problem}
An easier subproblem would be to show that as $a$ goes to infinity, this limit goes to zero.  Consider the rational polyhedral cone giving the condition that an $a$-core comes from a semigroup and intersect it with the region where the quadratic function $F_a(x_1,\ldots, x_{a-1}) \le 1$.  It seems likely that techniques from Ehrhart theory combined with the volume of this set can be used to solve this problem. 

%It seems likely that techniques from Ehrhart theory can be used to solve this problem along with a computation of the spherical volume of the rational polyhedral cone giving the condition that an $a$-core comes from a semigroups intersected with the region where the quadratic function $F_a(x_1,\ldots, x_{a-1}) \le 1$.

	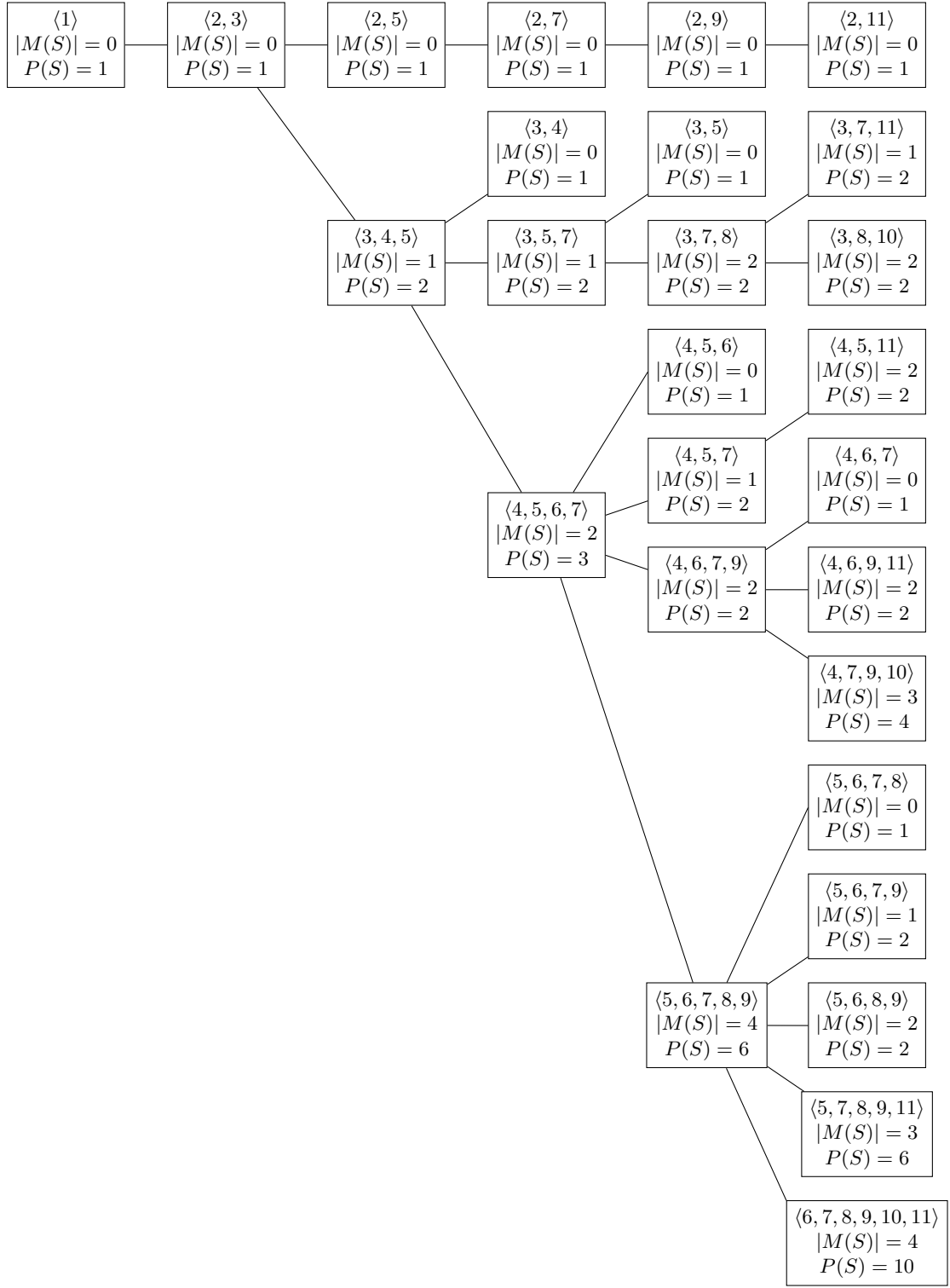
\begin{figure}[p] \label{semigroupTree}
		\centering
		\caption{The semigroup tree, with the root $\bb{N}$ on the left.
			Semigroups with a common genus are found in the same column, and each semigroup $S$ is labeled with $|M(S)|$ and $P(S)$.}
		\begin{tikzpicture}[xscale=2.5,yscale=1.7,
				font=\small]
			% genus 0
			\node[draw,align=center] (1) at (0,0)
				{$\langle 1 \rangle$ \\ $|M(S)|=0$ \\ $P(S) = 1$};
			% 1
			\node[draw,align=center] (23) at (1,0)
				{$\langle 2,3 \rangle$ \\ $|M(S)|=0$ \\ $P(S) = 1$}
				edge (1);
			% 2
			\node[draw,align=center] (25) at (2,0)
				{$\langle 2,5 \rangle$ \\ $|M(S)|=0$ \\ $P(S) = 1$}
				edge (23);
			\node[draw,align=center] (345) at (2,-2)
				{$\langle 3,4,5 \rangle$ \\ $|M(S)|=1$ \\ $P(S) = 2$}
				edge (23);
			% 3
			\node[draw,align=center] (27) at (3,0)
				{$\langle 2,7 \rangle$ \\ $|M(S)|=0$ \\ $P(S) = 1$}
				edge (25);
			\node[draw,align=center] (34) at (3,-1)
				{$\langle 3,4 \rangle$ \\ $|M(S)|=0$ \\ $P(S) = 1$}
				edge (345);
			\node[draw,align=center] (357) at (3,-2)
				{$\langle 3,5,7 \rangle$ \\ $|M(S)|=1$ \\ $P(S) = 2$}
				edge (345);
			\node[draw,align=center] (4567) at (3,-4.5)
				{$\langle 4,5,6,7 \rangle$ \\ $|M(S)|=2$ \\ $P(S) = 3$}
				edge (345);
			% 4
			\node[draw,align=center] (29) at (4,0)
				{$\langle 2,9 \rangle$ \\ $|M(S)|=0$ \\ $P(S) = 1$}
				edge (27);
			\node[draw,align=center] (35) at (4,-1)
				{$\langle 3,5 \rangle$ \\ $|M(S)|=0$ \\ $P(S) = 1$}
				edge (357);
			\node[draw,align=center] (378) at (4,-2)
				{$\langle 3,7,8 \rangle$ \\ $|M(S)|=2$ \\ $P(S) = 2$}
				edge (357);
			\node[draw,align=center] (456) at (4,-3)
				{$\langle 4,5,6 \rangle$ \\ $|M(S)|=0$ \\ $P(S) = 1$}
				% edge (4567)
				;
			\draw (4567) -- (456.west);
			\node[draw,align=center] (457) at (4,-4)
				{$\langle 4,5,7 \rangle$ \\ $|M(S)|=1$ \\ $P(S) = 2$}
				edge (4567);
			\node[draw,align=center] (4679) at (4,-5)
				{$\langle 4,6,7,9 \rangle$ \\ $|M(S)|=2$ \\ $P(S) = 2$}
				edge (4567);
			\node[draw,align=center] (56789) at (4,-9)
				{$\langle 5,6,7,8,9 \rangle$ \\ $|M(S)|=4$ \\ $P(S) = 6$}
				edge (4567);
			% 5
			\node[draw,align=center] (211) at (5,0)
				{$\langle 2,11 \rangle$ \\ $|M(S)|=0$ \\ $P(S) = 1$}
				edge (29);
			\node[draw,align=center] (3711) at (5,-1)
				{$\langle 3,7,11 \rangle$ \\ $|M(S)|=1$ \\ $P(S) = 2$}
				edge (378);
			\node[draw,align=center] (3810) at (5,-2)
				{$\langle 3,8,10 \rangle$ \\ $|M(S)|=2$ \\ $P(S) = 2$}
				edge (378);
			\node[draw,align=center] (4511) at (5,-3)
				{$\langle 4,5,11 \rangle$ \\ $|M(S)|=2$ \\ $P(S) = 2$}
				edge (457);
			\node[draw,align=center] (467) at (5,-4)
				{$\langle 4,6,7 \rangle$ \\ $|M(S)|=0$ \\ $P(S) = 1$}
				edge (4679);
			\node[draw,align=center] (46911) at (5,-5)
				{$\langle 4,6,9,11 \rangle$ \\ $|M(S)|=2$ \\ $P(S) = 2$}
				edge (4679);
			\node[draw,align=center] (47910) at (5,-6)
				{$\langle 4,7,9,10 \rangle$ \\ $|M(S)|=3$ \\ $P(S) = 4$}
				edge (4679);
			\node[draw,align=center] (5678) at (5,-7)
				{$\langle 5,6,7,8 \rangle$ \\ $|M(S)|=0$ \\ $P(S) = 1$}
				% edge (56789)
				;
			\node[draw,align=center] (5679) at (5,-8)
				{$\langle 5,6,7,9 \rangle$ \\ $|M(S)|=1$ \\ $P(S) = 2$}
				edge (56789);
			\node[draw,align=center] (5689) at (5,-9)
				{$\langle 5,6,8,9 \rangle$ \\ $|M(S)|=2$ \\ $P(S) = 2$}
				edge (56789);
			\node[draw,align=center] (578911) at (5,-10)
				{$\langle 5,7,8,9,11 \rangle$ \\ $|M(S)|=3$ \\ $P(S) = 6$}
				edge (56789);
			\node[draw,align=center] (67891011) at (5,-11)
				{$\langle 6,7,8,9,10,11 \rangle$ \\ $|M(S)|=4$ \\ $P(S) = 10$}
				% edge (56789)
				;
			\draw (56789) -- (67891011.north west);
			\draw (56789) -- (5678.west);
		\end{tikzpicture}
	\end{figure}

We would also like to better understand how to use techniques from the first part of this paper to study $P(S)$.  Suppose that $S$ is a numerical semigroup containing $a$.  Then every partition with hook set $S$ corresponds to a point in $\N^{a-1}$ by taking the Ap\'ery tuple of the corresponding numerical set.  Can we say anything meaningful about the geometry of this finite set of points?  We would also like to know the largest, smallest, and average size of a partition with hook set $S$.

We would also like to better understand the properties of $S$ that control the size of $P(S)$.  We have started to explore the link between the size of the set of missing pairs, $M(S)$, and the number of partitions with this hook set.  We include some data related to this question.  The \emph{semigroup tree} allows us to visualize easily the relationship between numerical semigroups via their \emph{effective generators}, the minimal generators greater than the Frobenius number.
The tree is constructed as follows: the vertices of the tree are numerical semigroups, with the root as $\N$; for each vertex $S$ in the tree, the children of this semigroup are the semigroups obtained from $S$ by removing an effective generator. Each semigroup appears in the tree exactly once, and the distance between $S$ and the root is exactly the genus of $S$. For more information about the semigroup tree, see \cite{Bras}.

Figure 2 shows the first 6 layers of the semigroup tree, in which each semigroup $S$ is labeled with $|M(S)|$ and $P(S)$.  Every semigroup generated by two elements is symmetric, so we see that these all satisfy $|M(S)| = 0$ and $P(S) = 1$.  We also see that the semigroups $\langle g+1,g+2,\ldots, 2g+1\rangle$ are those which have the largest values of $P(S)$ at a given genus.

%One can see how for semigroups generated by exactly two elements, $M(S)=0$ and $P(S)=1$, as expected since these semigroups are known to be symmetric.
%We also begin to see the rough correlation between $M(S)$ and $P(S)$, which seems to be true for ``most'' semigroups, despite those exemplified in Proposition \ref{missingPairFailure}.

Lastly, throughout this paper we have explored the properties of hook sets of partitions, but have not really commented on hook multisets.  We would like to better understand what properties of a multiset make it occur as the hook multiset of many different partitions.  A good starting place might be a careful examination of the constructions given by Chung and Herman \cite{Chung}, and by Craven \cite{Craven}.

%Multisets are important because of their relation to the theorem of Frame, Robinson, and Thrall, which gives the dimension of the representation of $S_n$ corresponding to some partition, and many of these areas that we have explored can be framed in terms of multisets instead.
%For more information on hook multisets, see \cite{Chung}.

% % % % % % % % % % % % % % % % % % % % % % % % % % % % % % % %

\section{Acknowledgments}
The third author thanks Mel Nathanson for organizing the 2010 CANT conference where he first learned about core partitions and their connection to numerical semigroups in a talk by William Keith.  He also thanks Maria Monks Gillespie for helpful discussions on early parts of this project.

We would like to thank Florencia Orosz-Hunziker and Daniel Corey for their assistance throughout this project.  Finally, we would like to thank the Summer Undergraduate Math Research at Yale program for organizing, funding, and supporting this project.  SUMRY is supported in part by NSF grant CAREER DMS-1149054.

%%%%%%%%%%%%%%%%%%%%%%%%%%%%%%%%%%%%%%%%%%%%%%%%%%%%%%%%%%%%%%%%%%%%%

\end{document}